\documentclass[12pt, a4paper]{amsart}

\usepackage[active]{srcltx} 
\usepackage{verbatim}
\usepackage{amsmath,amssymb,amsfonts,amsthm}
\usepackage{a4wide,enumerate}
\usepackage[latin1]{inputenc}
\usepackage{latexsym}
\usepackage{times,url}
\usepackage{graphicx,graphics}
\usepackage{epsfig}
\usepackage{epsf}
\usepackage{indentfirst}
\usepackage{hyperref}
\hypersetup{colorlinks=true,citecolor=blue,urlcolor=blue}
\newcommand{\R}{\mathbb R}   

\newcommand{\D}{\displaystyle}

\newtheorem{thm}{Theorem} [section]

\newtheorem{rem}[thm]{Remark}
\newtheorem{lem}[thm]{Lemma}

\newtheorem{cor}[thm]{Corollary}

\catcode`@=11 \@addtoreset{equation}{section} \catcode`@=12

\def\qq#1{\qquad \mbox{#1}\quad}

\newcommand{\al}{\alpha}

\newcommand{\de}{\delta}

\newcommand{\e}{\varepsilon}

\newcommand{\g}{\gamma}
\newcommand{\G}{\Gamma}

\newcommand{\la}{\lambda}

\newcommand{\om}{\omega}
\newcommand{\Om}{\Omega}

\newcommand{\p}{\partial}

\begin{document}

\title[]{A priori bounds for positive solutions of subcritical elliptic equations}

\author[A. Castro, R. Pardo]{Alfonso Castro, Rosa Pardo}

\address{A. Castro,
Department of Mathematics, Harvey Mudd College, Claremont, CA 91711, USA.} \email{{\tt castro@math.hmc.edu }}

\address{R. Pardo,
Departamento de Matem\'atica Aplicada, Universidad Complutense de Madrid, 28040--Madrid, Spain.} \email{{\tt rpardo@mat.ucm.es }}

\thanks{This work was partially supported by a grant from the  Simons Foundations (\# 245966 to Alfonso Castro). The second author is supported Spanish Ministerio de Ciencia e Innovacion (MICINN) under
Project MTM2012-31298. This work was started during a sabbatical visit of the second author to the Department of Mathematics, Harvey Mudd College, Claremont, USA, whose hospitality she thanks.}

\date{}

\begin{abstract}
We provide a-priori $L^\infty$ bounds for positive solutions to a class of subcritical elliptic problems in  bounded  $C^2$ domains.
Our arguments rely on the  moving planes method applied on the Kelvin transform of solutions.  We prove that  locally the image  through the inversion map of a neighborhood of the boundary contains a convex neighborhood; applying the moving planes method, we prove that  the transformed functions have  no  extremal   point  in a  neighborhood of the boundary of the inverted domain. Retrieving the original solution $u$,  the maximum of any positive solution in the domain $\Om,$ is bounded above by a constant  multiplied by the maximum on an open subset strongly contained in $\Om.$ The constant and the open subset  depend only on geometric properties of $\Om,$ and are independent of the non-linearity and on the solution  $u$. Our analysis answers a longstanding open problem.
\end{abstract}


\maketitle


\section{Introduction}

We provide a-priori $L^\infty(\Om)$ bounds for  a classical positive solutions to  the  boundary-value problem:
\begin{equation}
\label{eq:elliptic:problem} \left\{ \begin{array}{rcll} -\Delta u
&=&f(u), & \qquad \mbox{in } \Omega, \\ u&=& 0, & \qquad \mbox{on } \partial \Omega,
\end{array}\right.
\end{equation}
where $\Omega \subset \R ^N $,  $N\geq 2,$ is a bounded $C^{2}$ domain, and $f$ is a subcritical nonlinearity. For simplicity we assume $N>2,$ but our techniques fits well to the case $N=2.$ Our main result is:

\begin{thm}\label{th:apriori}
Assume that $\Omega \subset \R ^N $ is a bounded  domain  with $C^{2}$  boundary.
Assume that the nonlinearity $f$ is locally Lipschitzian and satisfies the following conditions
\begin{enumerate}
\item[\rm (H1)]   $\dfrac{f(s)}{s^{N^\star }}$ is  nonincreasing for any $s\geq 0,$ where  $N^\star =\frac{N+2}{N-2}$,
  \item[\rm (H2)] $f$ is {\it subcritical}, i.e.
  $\D\lim_{s\to\infty} \dfrac{f(s)}{s^{N^\star }}=0,$
  \item[\rm (H3)]
  $\D\liminf_{s\to\infty} \dfrac{f(s)}{s}>\lambda _1,$
where $\lambda_1$ is the first eigenvalue of $-\Delta$ acting on $H^1_0(\Om).$
\end{enumerate}
Then there exists a uniform constant $C ,$ depending only on $\Omega$ and $f,$ such that for every $u>0,$  classical solution to \eqref{eq:elliptic:problem},
$$\|u\|_{L^\infty (\Om)}\leq C.
$$
\end{thm}

Theorem \ref{th:apriori} answers a longstanding open problem, raised for instance in  \cite{Figueiredo-Lions-Nussbaum} as well as in \cite{Gidas-Spruck}.  Our analysis substantially  extends previous  results. In \cite{Figueiredo-Lions-Nussbaum} the nonlinearity $f$ is assumed to satisfy
$$\limsup_{s \to +\infty} (sf(s) - \theta F(s))/(s^2f^{N/2}(s)) \leq 0\qq{for some}\theta \in [0,2N/(N-2)),$$
where $F(s)=\int_0^t f(s)ds.$
The results in \cite{Gidas-Spruck} depend heavily on the {\it blow up} method which requires $f$ to be essentially of the form $f=f(x,s) = h(x)s^p$ with $p \in (1, N^*)$ and $h(x)$ continuous and strictly positive. Functions such as $f_1(s) = s^{N^*}/(\ln(s+2)$ satisfy our hypotheses but not those of \cite{Figueiredo-Lions-Nussbaum} neither of \cite{Gidas-Spruck}.

\smallskip

Next we provide an example of a nonlinearity $f$ that satisfies our hypotheses but not those of \cite{Gidas-Spruck}. Let $1 < p < q < N^*$.  Let $a_1$ be any real number larger than 1. Inductively we define $b_j=a_j^{(N^*-p)/(N^*-q)},$ and $a_{j+1} = b_j^{q/p}.$ Thus $a_j\leq b_j\leq a_{j+1}$ and
$\{a_j\},\{b_j\}$ are increasing sequences converging to $+\infty$.  We define $f(s) = s^p$ for $s\in [0,a_1]$.   Inductively, we define $f$ on
$[a_j, b_{j}]\cup [b_j, a_{j+1}]$  for $j = 1, 2, \ldots$ in the following way: $f(s) =   s^{N^*}/a_j^{N^*-p}$ for $s \in [a_j, b_j]$ and
$f(s) = f\left(b_j\right)$ for $s \in [b_j, a_{j+1}]$. It is easily seen that $s^p \leq f(s)\leq s^q$ for all $s \geq 1$. Hence
$f$ satisfies (H2) and (H3). Since $f$ is a multiple of $s^{N^*}$ on $[a_j, b_j]$, $f(s)/s^{N^*}$ is constant in that interval. On the the other hand, in $[b_j, a_{j+1}]$, $f$ is constant. Hence, in $[b_j, a_{j+1}]$, $f(s)/s^{N^*}$ decreases. Thus hypothesis (H1) is satisfied.
Since $f(a_j) = a_j^p$ and $f(b_j) = b_j^q,$ there is no $\al \in (1, N^*)$ such that $\lim_{s \to +\infty} f(s)/s^{\al} \in \mathbb{R}$.
Thus $f$ does not satisfy the hypotheses of Gidas-Spruck (see \cite[Theorem 1.1]{Gidas-Spruck}).

\medskip

Our proof of the Theorem \ref{th:apriori} uses  {\it moving plane arguments}, as in \cite{Figueiredo-Lions-Nussbaum}, as well as {\it Kelvin transform}. For the sake of completeness in the presentation,  below  we define the Kelvin transform, and in section 2 we recall results on moving plane arguments to be applied in section 3 in the proof of Theorem \ref{th:apriori}.

\smallskip

Applying the Kelvin transform to positive solutions of \eqref{eq:elliptic:problem}, the moving planes method
determines regions where the transformed function has no critical point. Recovering then the solution $u$, one sees that its maximum in the entire domain $\Om,$ is bounded above by a constant $C$ multiplied by the maximum of the same solution on an open subset $\om$ strongly contained in $\Om.$ The constant $C$ and the open subset  $\om\subset\subset \Om,$ depend only on geometric properties of $\Om,$ and they are independent of $f$ and $u,$ see Theorem \ref{th:comp}. This Theorem is a compactification process of a local version given earlier in  Theorem \ref{th:bdd:int}.

\medskip

The moving planes method was used earlier by Serrin in \cite{Serrin}.
For second order elliptic equations with spherical symmetry satisfying over-determined boundary conditions, he proved that positive solutions exists only when the domain is a ball and the solution is  spherically symmetric.
The proof is based on Maximum Principle and the moving planes method, which basically moves plains to a critical position, and then show that the solution is symmetric about this limiting plane.

\smallskip

Gidas-Ni and Nirenberg in \cite{Gidas-Ni-Nirenberg}, using this moving planes method and the Hopf Lemma, prove symmetry of positive solutions of elliptic equations vanishing on the boundary. See also Castro-Shivaji \cite{Castro-Shivaji}, where symmetry of nonnegative solutions is established for $f(0) < 0$. In \cite{Gidas-Ni-Nirenberg} the authors also characterized regions inside of $\Omega,$ next to the convex part of the boundary, where a positive solution cannot  have  critical points. Those regions depend only on the local convexity of $\Omega,$ and are independent of $f$ and $u.$ This non-existence of critical points in a whole region, is due to a strict monotonicity property of any positive solution in the normal direction.
This  is a key point to reach our results.

\smallskip

Gidas, Ni and Nirenberg in their classical paper pose the following problem 33 years ago which to the knowledge of the authors, is still open, see \cite[p. 223]{Gidas-Ni-Nirenberg}.

\smallskip

{\it Problem: Suppose $u>0$ is a classical solution of \eqref{eq:elliptic:problem}. Is there some $\e>0$ only dependent on the geometry of $\Om$ (independent of $f$ and $u$) such that $u$ has no stationary points in a $\e$-neighborhood of $\p\Om$?}

\smallskip

This is  true in convex domains, and for $N=2,$ see \cite[Corollary 3 and p. 223]{Gidas-Ni-Nirenberg}.
The question is now what about non-convex domains with $N>2.$

\smallskip

Our contribution is the following one: there are some $C$ and $\de>0$  depending only on the geometry of $\Om$ (independent of $f$ and $u$) such that
\begin{equation}\label{ineq:Om:de}
\max_{\Om} u\leq C \ \max_{\Om_{\de} } u
\end{equation}
where $\Om_{\de}:=\{x\in\Om\ : \ d(x,\p\Om)>\de\},$
see Theorem \ref{th:comp}.

\medskip

To reach our answer in this situation, let us start by defining the Kelvin transform, see \cite[proof of theorem 4.13, p. 66-67]{G-T}.

Let us recall that every $C^2$ domain  $\Om$  satisfy the following condition, known as the {\it uniform exterior sphere condition},
\begin{itemize}
\item[(P)] there exists a $\rho>0$ such that for every $x\in\p\Om$ there exists a ball $B=B_{\rho} (y)\subset\mathbb{R}^N\setminus\Om$ such that $\p B\cap\p\Om=x.$
\end{itemize}

\smallskip

Let $x_0\in \p\Om,$ and let $\overline{B}$ be the closure of a ball intersecting $\overline{\Om}$ only at the point $x_0.$
Let us  suppose $x_0=(1,0,\cdots,0),$ and $B$ is the unit ball with center at the origin. The {\it inversion mapping}
\begin{small}
\begin{equation}\label{def:h}
x\to h(x)=\frac{x}{|x|^2},
\end{equation}
\end{small}
is an homeomorphism from $\R^N\setminus\{0\}$ into itself. We perform an inversion  from $\Om$  into the unit ball $B,$ in terms of the inversion map $ h\left|_\Om\right. ,$  see fig. \ref{fig3_4_5} (a).

\begin{figure}[tb]
$$\kern -2em
\begin{array}{ccc}
\includegraphics[width=6cm]{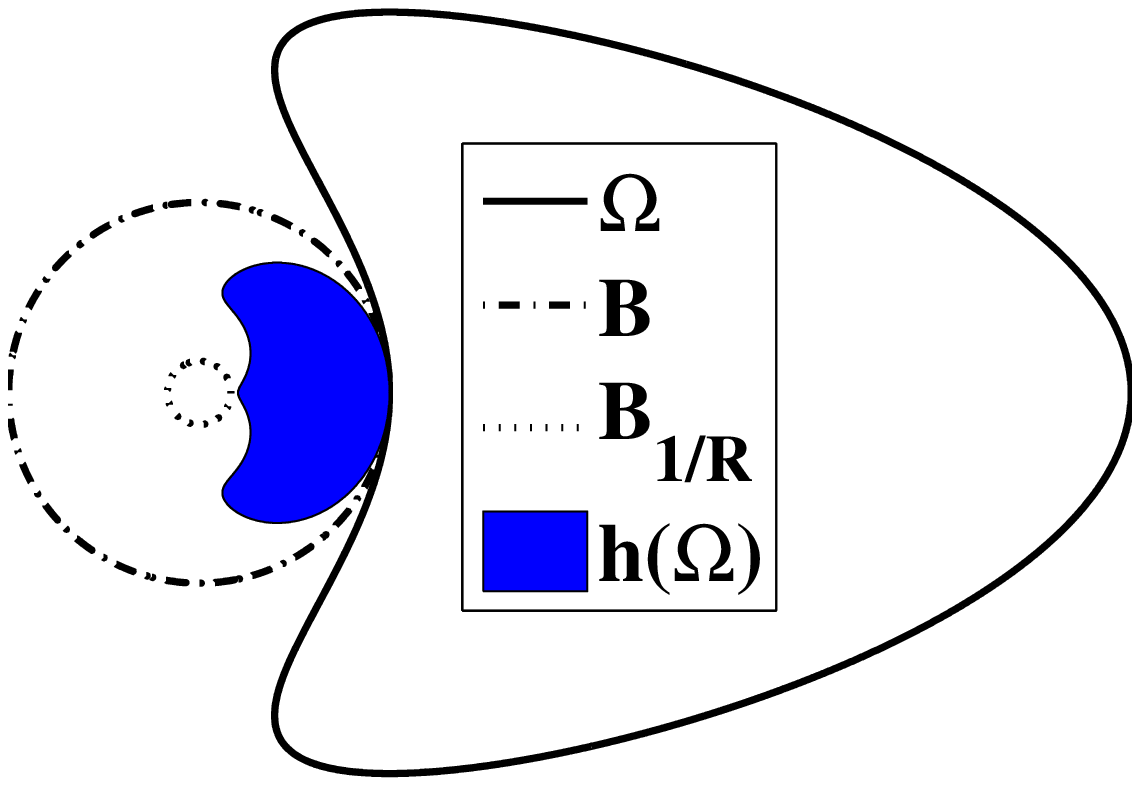} &
\includegraphics[width=2cm]{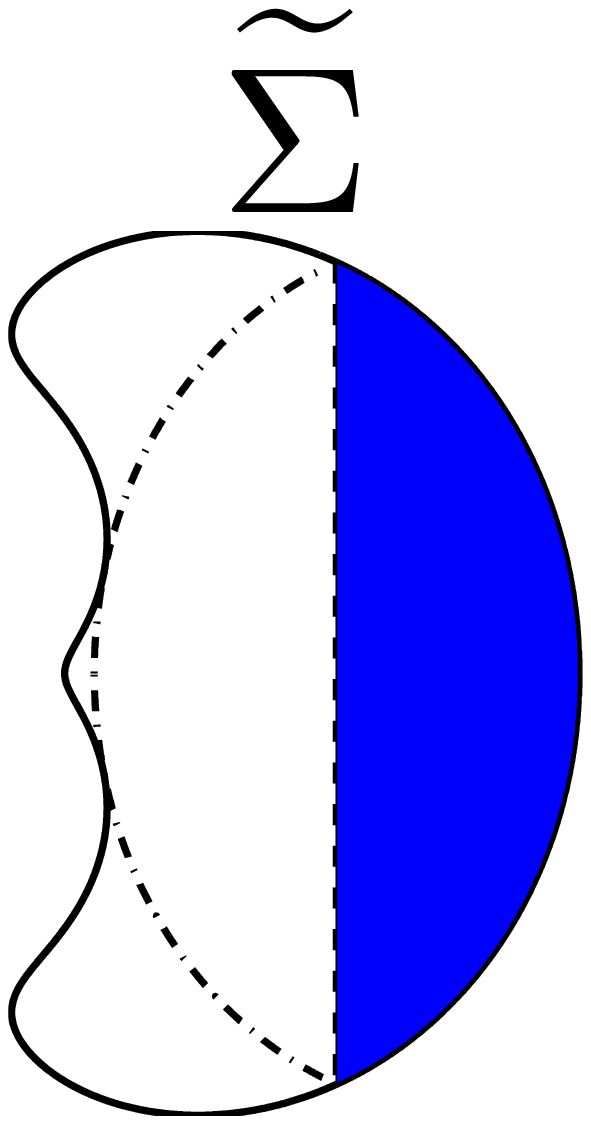} &
\includegraphics[width=6cm]{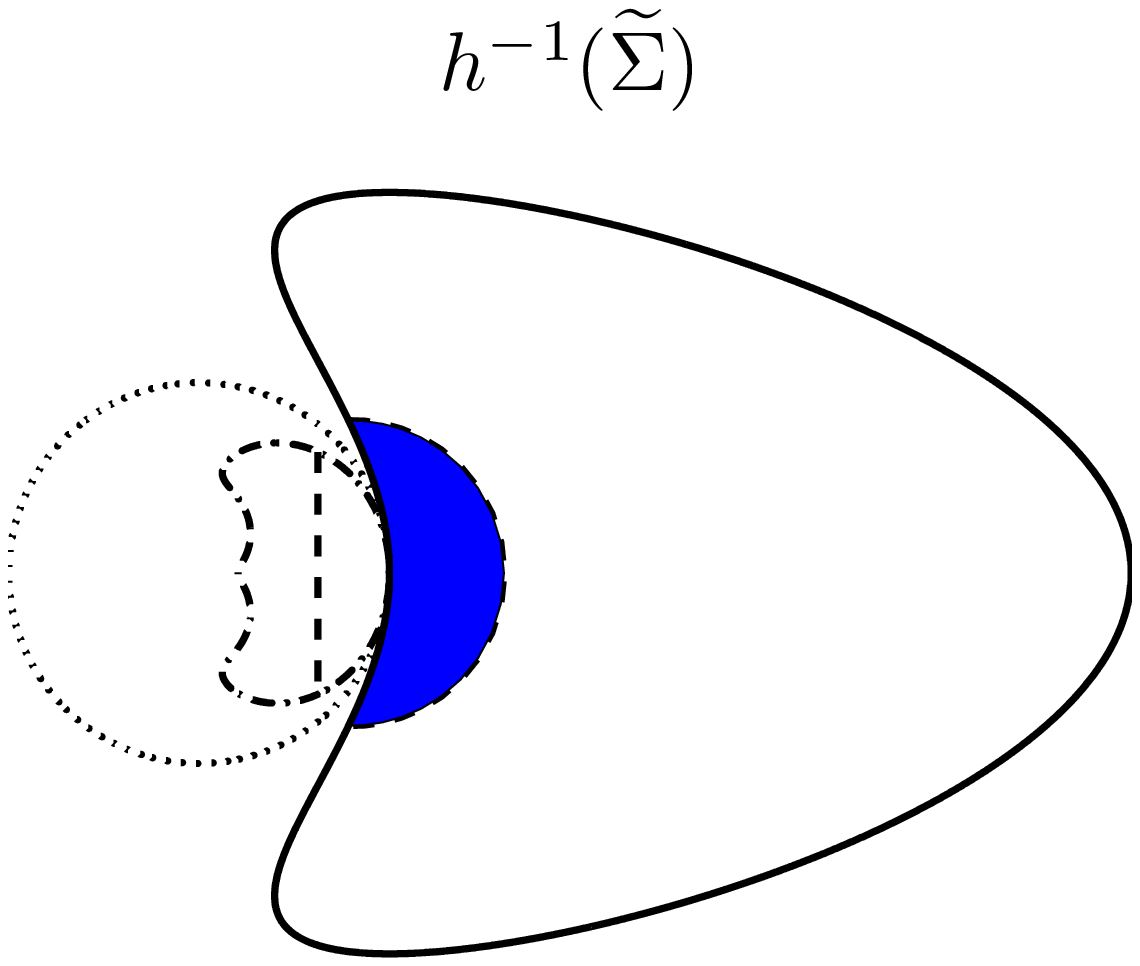} \\
 {\rm (a)} & {\rm (b)} & {\rm (c)}
\end{array}
$$
\label{fig3_4_5}
\caption{
(a) The exterior tangent ball and the inversion of the boundary into the unit ball.
(b) A maximal cap $\widetilde{\Sigma}$ in the transformed domain $h(\Omega )$.
(c) The set $h^{-1}(\widetilde{\Sigma})$ (i.e. the inverse image of the maximal cap $\widetilde{\Sigma}$) in the original domain $\Omega $.}
\end{figure}

Let $u$ solve \eqref{eq:elliptic:problem}. The  {\it Kelvin transform} of $u$ at the point $x_0\in\p\Om$ is defined in the transformed domain $\widetilde{\Om}:=h(\Om)$ by
\begin{small}
\begin{equation}\label{v}
v(y):=\left(\frac{1}{|y|}\right)^{N-2}\ u\left(\frac{y}{|y|^2}\right),\qq{for}y\in\widetilde{\Om}.
\end{equation}
\end{small}
We first prove that, for each point $x_0\in\p\Om,$ there exists some $\delta>0$ depending only on  the geometry of $\Om,$ (independent of $f$ and $u$), such that  its Kelvin transform has no stationary point in  $B_\de(x_0)\cap h(\Om)$, see Theorem \ref{th:kelvin}.

\smallskip

Retrieving the solution $u$ of \eqref{eq:elliptic:problem} we obtain that
$$
\max_{\Om} u\leq C \ \max_{\Om\setminus B_{\de'}(x_0) } u
$$
where  $C$ only depends on $\Om$ and it is independent of $f$ and $u$, see Theorem \ref{th:bdd:int}.

Next, we move $x_0\in\p\Om$ obtaining \eqref{ineq:Om:de},
see Theorem \ref{th:comp}.

\medskip

This paper is organized in the following way. In Section \ref{sec:mov:pl} we describe the moving planes method, and its consequences when applied to the Kelvin transform of the solution. In particular Theorem \ref{th:kelvin}, Theorem \ref{th:bdd:int}, and Theorem \ref{th:comp} are included in this section.
In Section \ref{sec:apriori} we prove our main results on a-priori bounds, see Theorem \ref{th:apriori}.
We include an Appendix with geometrical results on the local convexity of the inverted image of the domain, see Lemma \ref{lem:convex}.

\section{The moving planes method and the Kelvin transform}
\label{sec:mov:pl}

We first collect some well known results on the moving planes method:  Theorem \ref{th:moving:pl}  and Theorem \ref{th:moving:pl:nonl}. Next, we state our main results in this section: Theorem \ref{th:kelvin}, Theorem \ref{th:bdd:int}, and Theorem \ref{th:comp}.


We next expose the moving planes method.
We will be moving planes in the $x_1$-direction to fix ideas. Let us first define some  concepts and notations.
\begin{trivlist}
\item[-] The {\it moving plane} is defined in the following way: $ T_\lambda := \{ x\in \mathbb{R}^N : x_1 = \lambda \} ,$

\item[-] the {\it cap}:  $\quad \Sigma_\lambda := \{ x=(x_1,x')\in \mathbb{R}\times\mathbb{R}^{N-1}\cap \Omega \ :\ x_1 < \lambda \} ,$

\item[-] the {\it reflected point}: $\quad x^\lambda := (2\lambda - x_1,x') ,$

\item[-] the {\it reflected cap}: $\quad \Sigma'_\lambda := \{x^\lambda \ :\ x\in\Sigma_\lambda \},$   see fig. \ref{fig1_2}(a).

\item[-] the minimum value for $\lambda$ or starting value: $\quad \lambda_0 := \min \{ x_1\ :\ x\in \Omega \},$

\item[-] the maximum value for $\lambda$: $\quad \lambda^\star  := \max \{ \lambda\ :\ \Sigma'_{\mu} \subset \Omega \qq{for all} \mu\leq\lambda \},$

\item[-] the {\it maximal cap}: $\quad \Sigma:=\Sigma_{\lambda^\star  }.$
\end{trivlist}
\begin{figure}[tb]
$$\kern -2em
\begin{array}{ccc}
\includegraphics[width=4.5cm]{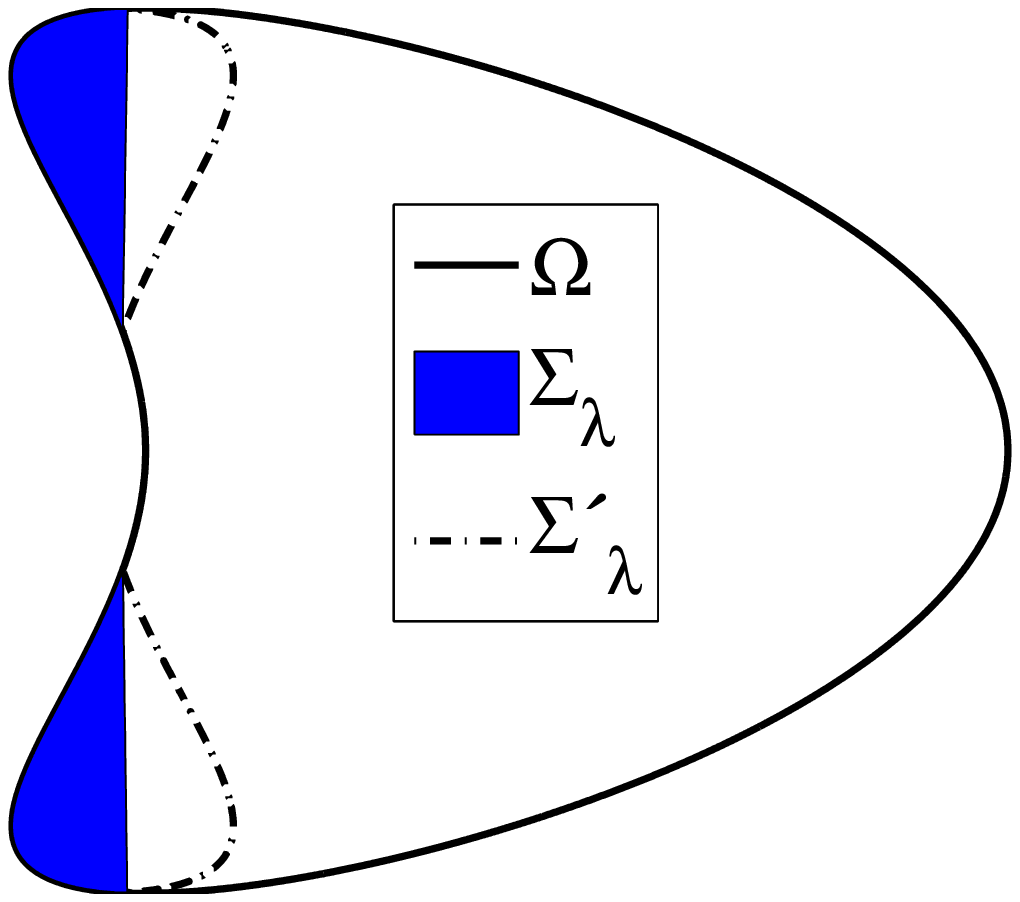} &
\includegraphics[width=4.5cm]{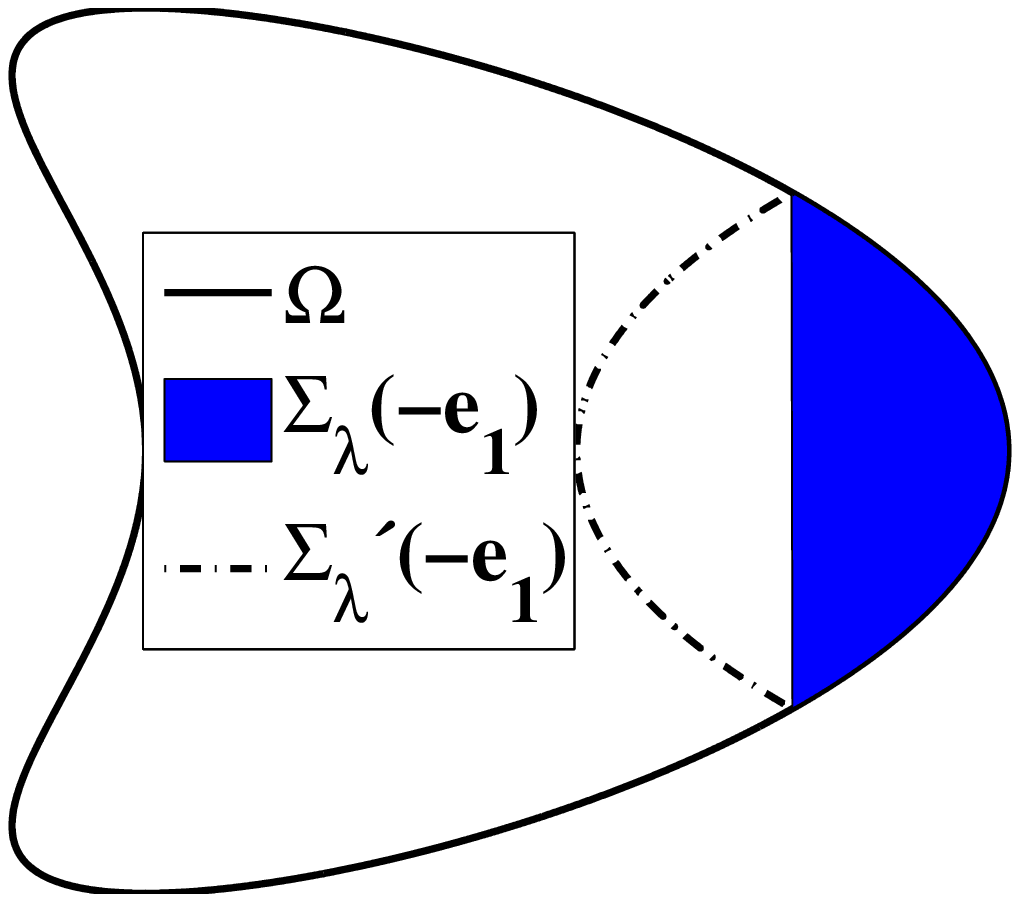} &
\includegraphics[width=4.5cm]{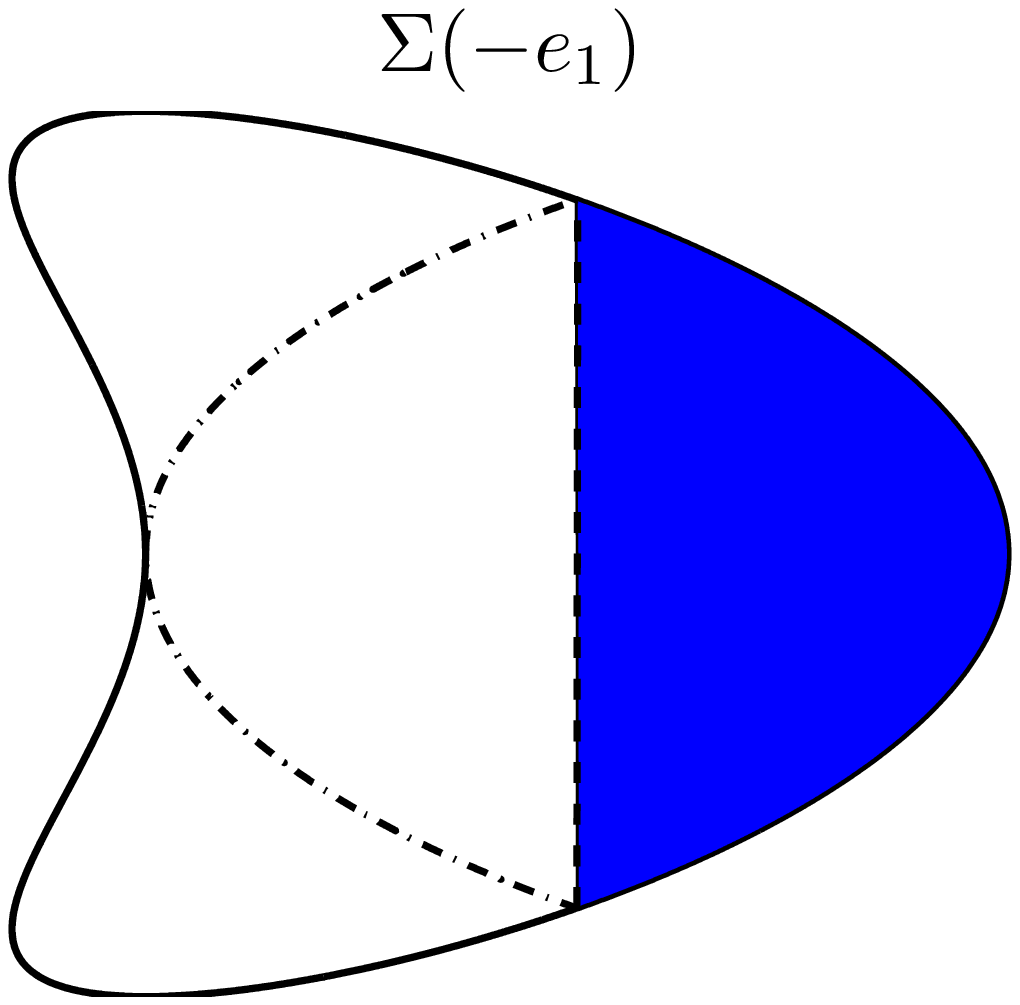} \\
{\rm (a)} & {\rm (b)} & {\rm (c)}
\end{array}
$$
\label{fig1_2}
\caption{(a) A  cap $\Sigma_\la$ and its reflected cap $\Sigma'_\la$  in the $e_1$ direction.
(b) A  cap $\Sigma_\la(-e_1)$ and its reflected cap $\Sigma'_\la(-e_1)$ (in the $-e_1$ direction).
(c) A maximal cap $\Sigma(-e_1)$.}
\end{figure}

The following Theorem is Theorem 2.1 in  \cite{Gidas-Ni-Nirenberg}. \begin{thm}\label{th:moving:pl}
Assume that $f$ is locally Lipschitz, that $\Omega$ is bounded and
that $T_\lambda,$ $x^\lambda,$ $\lambda_0,$ $\lambda^\star ,$  $\Sigma_\lambda $ $\Sigma'_\lambda ,$ and $\Sigma$ are as above. If $u\in C^2(\overline{\Om})$ satisfies  \eqref{eq:elliptic:problem} and $u>0$ in $\Om$, then for any $\lambda\in (\lambda_0,\lambda^\star )$
$$u (x) < u(x^\lambda)  \qq{and} \frac{\p u\ \ }{\p x_1}(x) > 0 \qq{for all} x\in\Sigma_\lambda .$$
Furthermore, if $\ \frac{\p u\ \ }{\p x_1}(x) = 0$ at some point in $\Om \cap T_{\lambda^\star },$
then necessarily $u$ is symmetric  in the plane $T_{\lambda^\star },$ and $\Om=\Sigma\cup \Sigma'\cup (T_{\lambda^\star } \cap\Om).$
\end{thm}
\begin{proof}
See  \cite[Theorem 2.1 and Remark 1, p.219]{Gidas-Ni-Nirenberg} for $f\in C^1$
and   locally  Lipschitzian respectively.
\end{proof}

\begin{rem}{\rm
Set $x_0\in \p\Om\cap T_{\lambda_0},$ see fig. \ref{fig1_2}(a). Let us observe that by definition of  $\lambda_0,$ $T_{\lambda_0}$ is the tangent plane to the graph of the boundary at $x_0$, and  the inward normal at $x_0,$  is $n_i(x_0)=e_1.$  The above Theorem says that  the partial derivative following the direction given by the inward normal at the tangency point is strictly positive in the whole maximal cap. Consequently, {\it there are no critical points in the maximal cap}}.
\end{rem}

Now, we apply the above Theorem in any direction. According to the above Theorem, any positive solution of \eqref{eq:elliptic:problem} satisfying {\rm (H1)} has no stationary point in any maximal cap moving planes in any direction. This is the statement of the following Corollary. First, let us fix the notation
for a general $\nu\in\mathbb{R}^N$ with $|\nu|=1.$ We set
\begin{trivlist}
\item[-] the {\it moving plane} defined as:  $\quad T_\lambda (\nu)= \{ x\in \mathbb{R}^N\ :\ x\cdot \nu = \lambda \} ,$
\item[-] the {\it cap}:     $\quad \Sigma_\lambda (\nu)= \{ x\in \Omega \ :\ x\cdot \nu < \lambda \} ,$
\item[-] the {\it reflected point}:
    $\quad x^\lambda (\nu)= x+2(\lambda- x\cdot \nu)\nu ,$
\item[-] the {\it reflected cap}:
    $\quad \Sigma'_\lambda (\nu)= \{x^\lambda \ :\ x\in\Sigma_\lambda (\nu)\},$ see fig. \ref{fig1_2}(b), for $\nu=-e_1,$
\item[-] the minimum value of $\lambda$:
    $\quad    \lambda_0 (\nu)= \min \{ x\cdot \nu \ :\ x\in \Omega \},$
\item[-] the maximum value of $\lambda$:
$\ \ \lambda^\star  (\nu)= \max \{ \lambda\, :\, \Sigma'_{\mu}(\nu) \subset \Omega \ \ \mbox{for all } \mu\leq\lambda \},$
\item[-] and the {\it maximal cap}: $\quad \Sigma(\nu)=\Sigma_{\lambda^\star  (\nu)}(\nu),$
see fig. \ref{fig1_2}(c), for $\nu=-e_1.$
\end{trivlist}
Finally, let us also define the {\it optimal cap set}
\begin{small}\begin{equation}\label{opt:caps}
\Om^\bigstar    = \bigcup_{\{ \nu\in\R^N,|\nu|=1\}} \Sigma(\nu).
\end{equation}\end{small}

Applying Theorem \ref{th:moving:pl} in any direction, we can assert that there are not critical points in the union of all the maximal caps following any direction. The set $\Om^\bigstar$ is the union of the maximal caps in any direction, and in particular, the maximum of a positive solution is attained in the complement of $\Om^\bigstar.$ Thus we have:

\begin{cor}\label{co}
Assume that $f$ is locally Lipschitzian, that $\Omega$ is bounded, and that $\Om^\bigstar$ is the optimal cap set defined as above.

If  $u\in C^2(\overline{\Om})$ satisfies  \eqref{eq:elliptic:problem} and $u>0$ in $\Om$ then
$$\max_{\overline{\Om}} u=\max_{\overline{\Om} \setminus \Om^\bigstar } u.
$$
\end{cor}

\smallskip

If $\Om^\bigstar $ is a full boundary neighborhood of $\p\Om$ in $\overline{\Om},$ as it happens in convex domains, then there is $\e>0$ depending only on the geometry of $\Om$ (independent of $f$ and $u$) such that $u$ has no stationary points in a $\e$-neighborhood of $\p\Om.$
Next we study the case in which $\Om^\bigstar $ is not a  neighborhood of $\p\Om$ in $\Om.$

\smallskip

We prove  that the maximum of $u$ in the whole domain $\Om$ can be bounded above by a constant multiplied by the maximum of $u$ in some open set strongly contained in $\Om,$  see Theorem \ref{th:comp} below.

\smallskip

To achieve this result,  we will need the moving plane method for a nonlinearity $f=f(x,u)$. Next we study this method on nonlinear equations in a more general setting. Let us consider the nonlinear equation
\begin{equation}
\label{eq:elliptic:problem:nonl}
F\left(x,u,\nabla u,\left(\p^2_{ij}u\right)_{i,j=1,\cdots,N}\right)=0,
\end{equation}
where $F:\Om\times \mathbb{R}\times \mathbb{R}^N\times \mathbb{R}^{N\times N}$ is a real function, $F=F(x,s,p,r)$ and $\p^2_{ij}u=\frac{\p^2 u}{\p x_{i}\p x_j}.$ The operator $F$ is assumed to be elliptic, i.e. for positive constants $m,\ M$
\begin{equation*}
M|\xi|^2\geq  \sum_{i,j}\frac{\p F}{\p r_{ij}}\xi_i\xi_j\geq m|\xi|^2,\qquad \forall \xi\in\mathbb{R}^N.
\end{equation*}

On the function $F$  we will  assume
\begin{itemize}
  \item[\rm (F1)] $F$  is continuous and differentiable with respect to the  variables $s,p_i,r_{i,j},$ for all values of its arguments $(x,s,p,r)\in\overline{\Om}\times \mathbb{R}\times \mathbb{R}^N\times \mathbb{R}^{N\times N}.$
  \item[\rm (F2)] For all $x\in\p\Om \cap \{x_1<\lambda^\star \},$ $F(x,0,0,0)$ satisfies either
  $$F(x,0,0,0)\geq 0 \qq{or} F(x,0,0,0)< 0  .$$
  \item[\rm (F3)] $F$ satisfies
  $$F\left(x^\lambda,s,(-p_1,p'),\hat{r}\right)\geq F(x,s,p,r),
  $$
  for all $\lambda\in [\lambda_0,\lambda^\star ),\ x\in\Sigma(\lambda)$ and $(s,p,r)\in \mathbb{R}\times \mathbb{R}^N\times \mathbb{R}^{N\times N}$ with $s>0$ and $p_1<0,$ where $p=(p_1,p')\in \mathbb{R}\times \mathbb{R}^{N-1},$ $\hat{r}=\begin{small}\left(
                          \begin{array}{cc}
                            r_{11} & -r_{1\cdot}' \\
                            r_{21} & r_{2\cdot}' \\
                            \vdots & \vdots \\
                            r_{N1} & r_{N\cdot}' \\
                          \end{array}\right)\end{small}
                          ,$ and $r_{i\cdot}':=(r_{i2},\cdots,r_{iN}),$ for $i=1,\cdots,N.$
\end{itemize}

The following theorem is Theorem 2.1' in \cite{Gidas-Ni-Nirenberg}.
\begin{thm}\label{th:moving:pl:nonl}
Assume that  $\Omega$ is bounded and
that $T_\lambda,$ $x^\lambda,$ $\lambda_0,$ $\lambda^\star ,$  $\Sigma_\lambda $ $\Sigma'_\lambda ,$ and $\Sigma$ are as above.
Let $F$ satisfies conditions {\rm (F1)}, {\rm (F2)} and {\rm (F3)}.

If $u\in C^2(\overline{\Om})$ satisfies  \eqref{eq:elliptic:problem:nonl} and $u>0$ in $\Om$, then for any $\lambda\in (\lambda_0,\lambda^\star )$
$$u (x) < u(x^\lambda)  \qq{and} \frac{\p u\ \ }{\p x_1}(x) > 0 \qq{for all} x\in\Sigma_\lambda .$$
Furthermore, if
$\ \frac{\p u\ \ }{\p x_1}(x) = 0$ at some point in $\Om \cap T_{\lambda^\star },$
then necessarily $u$ is symmetric  in the plane $T_{\lambda^\star },$ and $\Om=\Sigma\cup \Sigma'\cup (T_{\lambda^\star } \cap\Om).$
\end{thm}

As an immediate Corollary in the semilinear situation we have the following one.
\begin{cor}\label{co2}
Suppose $u\in C^2(\overline{\Om})$ is a positive solution of
\begin{equation}\label{eq:f:x:u}
-\Delta u=f(x,u),  \quad \mbox{in } \Omega, \qquad u= 0,  \quad \mbox{on } \partial \Omega.
\end{equation}
Assume $f=f(x,s)$  and its first derivative $f_s$ are continuous,  for $(x,s)\in \overline{\Om}\times\R.$

Assume that
\begin{equation}\label{f:x:u}
    f(x^\lambda,s)\geq f(x,s)\qq{for all} x\in\Sigma(\lambda^\star ),\qq{for all} s>0.
\end{equation}

Then for any $\lambda\in (\lambda_0,\lambda^\star )$
$$u (x) < u(x^\lambda)  \qq{and} \frac{\p u\ \ }{\p x_1}(x) > 0 \qq{for all} x\in\Sigma_\lambda .$$
Furthermore, if $\ \frac{\p u\ \ }{\p x_1}(x) = 0$ at some point in $\Om \cap T_{\lambda^\star },$
then necessarily $u$ is symmetric  in the plane $T_{\lambda^\star },$ and $\Om=\Sigma\cup \Sigma'\cup (T_{\lambda^\star } \cap\Om).$
\end{cor}

Next, we state the first of our main results in this section, fixing regions where the Kelvin transform of the solution has no critical points. This is the statement of the following Theorem.
Let us fix some notation.
For any $x_0\in\p\Om$, let $\tilde{n}_i(x_0)$ be the inward normal at $x_0$ in the transformed domain $\tilde{\Om}=h(\Om),$ where $h$ is defined in \eqref{def:h},
and let $\widetilde{\Sigma}= \widetilde{\Sigma} (\tilde{n}_i(x_0))$
be its maximal cap,  see fig. \ref{fig3_4_5}(b).

\begin{thm}\label{th:kelvin} Assume that $\Omega \subset \R ^N $ is a bounded  domain  with $C^{2}$  boundary. Assume that the nonlinearity $f$ satisfies {\rm (H1)} and {\rm (H2)}.

If $u\in C^2(\overline{\Om})$ satisfies  \eqref{eq:elliptic:problem} and $u>0$ in $\Om$, then for any $x_0\in\p\Om$ its maximal cap  in the transformed domain $\widetilde{\Sigma}$ is nonempty,
and its Kelvin transform $v,$ defined by \eqref{eq:kelvin}, has no critical point in the  maximal cap $\widetilde{\Sigma}.$

Consequently, for any $x_0\in\p\Om,$ there exists a $\delta>0$ only dependent of $\Om$ and $x_0,$ and independent of $f$ and $u$ such that  its Kelvin transform $v$ has no critical point in the set $B_\de(x_0)\cap h(\Om)$.
\end{thm}
\begin{proof}
Since  $\Om$ is a  $C^2$ domain, it  satisfies a uniform exterior sphere condition (P). Let $x_0\in \p\Om,$ and let $\overline{B}$ be the closure of a ball intersecting $\overline{\Om}$ only at the point $x_0.$
For convenience, by scaling, translating and rotating the axes, we may assume that $x_0=(1,0,\cdots,0),$ and $B$ is the unit ball with center at the origin.


We perform an inversion  $h$ from $\Om$  into the unit ball $B,$ by using the inversion map $x\to h(x)=\frac{x}{|x|^2}.$
Due to $\overline{B}\cap\overline{\Om}=\{x_0\}$, and to the boundedness of $\Om,$ there exists some $R>0,$ such that
\begin{equation}\label{bd:Om}
1\leq |x|\leq R \qq{for any} x\in\Om,
\end{equation}
and the image
\begin{equation}\label{Om:tilde}
\widetilde{\Om}=h(\Om)=\left\{y=h(x)\in\mathbb{R}^N\ :\ x=\frac{y}{|y|^2}\in\Om\right\} \subset B\setminus B_{1/R}.
\end{equation}
Note that $0\not\in h(\Om),$ see fig. \ref{fig3_4_5}(a).
Moreover $\widetilde{\Om}$ is strictly convex near $x_0$ and the maximal cap $\widetilde{\Sigma}= \widetilde{\Sigma} (\tilde{n}_i(x_0))$ contains a full neighborhood of $x_0$ in $\widetilde{\Om},$
where $\tilde{n}_i(x_0)$ is the normal inward at $x_0,$ see lemma \ref{lem:convex} in the Appendix,  see also fig. \ref{fig3_4_5}(b). Observe that, by construction $\tilde{n}_i(x_0)=-e_1.$

Next, we consider the Kelvin transform of the solution defined by \eqref{v}.
The function $v$ is well defined on $h(\Om),$ and writing $r = |x|,$ $\om =\frac{x}{|x|}$
and $\Delta_\om$ for the Laplace-Beltrami operator on $\p B_1,$ the function $v$ satisfies
\begin{small}
\begin{eqnarray*}
  \Delta v (r, \om ) &=& \left[ \frac{1}{r^{N-1}} \frac{\p}{\p r}\left( r^{N-1} \frac{\p}{\p r}\right)
+ \frac{1}{r^2}\Delta_\om\right] v (r, \om )
 \\
   &=& \left[\frac{1}{r^{N-1}} \frac{\p}{\p r}\left( r^{N-1} \frac{\p}{\p r}\right)
+\frac{1}{r^2}\Delta_\om\right] \left(\frac{1}{r}\right)^{N-2}u{\textstyle  \left(\frac{1}{r},\om\right)} \\
   &=&  \frac{1}{r^{N-1}} \frac{\p}{\p r} r^{N-1} \frac{\p}{\p r}\left[\left(\frac{1}{r}\right)^{N-2} u{\textstyle  \left(\frac{1}{r},\om\right)}\right]
   +\frac{1}{r^N}\Delta_\om u\\
   &=&  \frac{1}{r^{N-1}} \frac{\p}{\p r}  \left[-(N-2)u{\textstyle  \left(\frac{1}{r},\om\right)}- \frac{1}{r} u_r{\textstyle  \left(\frac{1}{r},\om\right)}\right]
   +\frac{1}{r^N}\Delta_\om u\\
   &=&  \frac{1}{r^{N-1}}   \left[\frac{(N-2)}{r^2}u_r+ \frac{1}{r^2} u_r+ \frac{1}{r^3} u_{rr}\right]
   +\frac{1}{r^N}\Delta_\om u\\
   &=&\frac{1}{r^{N+2}}\left[u_{rr}+\frac{N-1}{1/r} u_r+\frac{1}{1/r^2}\Delta_\om u\right]
   =\frac{1}{r^{N+2}} \Delta u {\textstyle  \left(\frac{1}{r},\om\right)}.
\end{eqnarray*}
\end{small}
Therefore $v>0$ in $\widetilde{\Om}$ satisfies
\begin{equation}\label{eq:kelvin}
-\Delta v(y) =\D\frac{1}{|y|^{N+2}} f\left(|y|^{N-2} v(y)\right), \quad \mbox{in } \widetilde{\Om}, \qquad v= 0,  \quad \mbox{on } \partial \widetilde{\Om}.
\end{equation}

From hypothesis {\rm (H2)}, we see that the function $g(y,s)=\frac{1}{|y|^{N+2}} f\left(|y|^{N-2} s\right)$ satisfies the hypothesis of Corollary \ref{co2}. By construction, it is straightforward that
$|y^\la|<|y|$ for all $y\in \widetilde{\Sigma},$
see fig. \ref{fig3_4_5} (a) and (b), and remain that the origin is at the center of the ball $B.$ By {\rm (H2)},
\begin{equation}\label{ineq:g}
g(y^\la,s)\geq g(y,s) \qq{for all} y\in \widetilde{\Sigma} ,
\end{equation}
where $\widetilde{\Sigma}$ is the maximal cap in the transformed domain, see fig. \ref{fig3_4_5}  (b).
Therefore,  the hypotheses of Corollary \ref{co2} are fulfilled, and hence $v$ has no critical point in the maximal cap $\widetilde{\Sigma}$, which completes the  proof choosing   $\delta$ such that $B_{\delta}(x_0)\cap h(\Om)\subset\widetilde{\Sigma}.$
\end{proof}

We are now ready to state our main result  in this section. This result is composed of two theorems, the first one, Theorem \ref{th:bdd:int} below is the local version in a neighborhood of  a boundary point, the second one, Theorem \ref{th:comp} is the global version.

\begin{thm}\label{th:bdd:int} Assume that $\Omega \subset \R ^N $ is a bounded  domain  with $C^{2}$  boundary. Assume that the nonlinearity $f$ satisfies {\rm (H1)} and {\rm (H2)}.
If $u\in C^2(\overline{\Om})$ satisfies  \eqref{eq:elliptic:problem} and $u>0$ in $\Om$, then for any $x_0\in\p\Om$ there exists a $\delta>0$ only dependent of $\Om$ and $x_0,$ and independent of $f$ and $u$ such that
\begin{equation}\label{bdd}
\max_{\Om} u\leq C \max_{\Om\setminus B_{\delta}(x_0)} u.
\end{equation}
The constant $C$  depends on $\Om$ but not on $x_0,$  $f$ or $u.$
\end{thm}

\begin{proof} Let $x_0\in\p\Omega,$  if   there exists a $\delta>0$ such that $ B_\delta (x_0)\cap \Om\subset \Om^\bigstar,$ (as it happens in convex sets), the proof follows from Theorem \ref{th:kelvin}. We concentrate our attention in the complementary set.

\smallskip

Let $x_0\in \p\Om,$ and let $\overline{B}$ be the closure of a ball intersecting $\overline{\Om}$ only at the point $x_0.$
Let  $v$ be as defined  in \eqref{v} for  $y\in\widetilde{\Om}=h(\Om).$
By a direct application of Theorem \ref{th:kelvin}, $v$  has no critical point in the maximal cap $\widetilde{\Sigma}$, and therefore
\begin{equation}\label{eq:kelvin:max}
\max_{\widetilde{\Om}} v(y)=\max_{\widetilde{\Om}\setminus \widetilde{\Sigma}} v(y).
\end{equation}
From definition of $v,$ see \eqref{v}, we obtain that
$$\max_{\Om} |x|^{N-2}u(x)=\max_{\Om\setminus h^{-1}(\widetilde{\Sigma})} |x|^{N-2}u(x),
$$
where $h^{-1}(\widetilde{\Sigma})$ is the inverse image of the maximal cap, see fig \ref{fig3_4_5}(b)-(c). Due to the boundedness of $\Om,$   see \eqref{bd:Om}, we deduce
$$\max_{\Om} u(x)\leq R^{N-2} \max_{\Om\setminus h^{-1}(\widetilde{\Sigma})} u(x),
$$
which concludes the proof choosing $C=R^{N-2}$ and  $\delta$ such that $B_{\delta}(x_0)\subset h^{-1}(\widetilde{\Sigma})$ and therefore $\Om\setminus h^{-1}(\widetilde{\Sigma})\subset \Om\setminus B_{\delta}(x_0).$
\end{proof}

The following Theorem  is just a compactification process of the above result.

\begin{thm}\label{th:comp} Assume that $\Omega \subset \R ^N $ is a bounded  domain  with $C^{2}$  boundary. Assume that the nonlinearity $f$ satisfies {\rm (H1)} and {\rm (H2)}.
If $u\in C^2(\overline{\Om})$ satisfies  \eqref{eq:elliptic:problem} and $u>0$ in $\Om$, then  there exists two constants $C$ and $\delta$ depending only on $\Om$ and not on $f$ or $u$ such that
\begin{equation}\label{bdd:2}
\max_{\Om} u\leq C \ \max_{\Om_{\delta} } u
\end{equation}
where $\Om_{\delta}:=\{x\in\Om\ : \ d(x,\p\Om)>\delta\}.$
\end{thm}

\begin{proof}
Since $\Om$ is a  $C^2$ domain, it  satisfies a uniform exterior sphere condition (P). Thanks to that property, we can choose a constant $C=(R/\rho)^{N-2}$ satisfying the above inequality.

Moreover, let us note that from Theorems \ref{th:kelvin} and \ref{th:bdd:int}, the constant $\delta$ only depends on geometric properties of the domain $\Om.$
\end{proof}

\section{Proof of Theorem \ref{th:apriori}}
\label{sec:apriori}
\begin{proof}[Proof of Theorem \ref{th:apriori}]
We shall argue by contradiction.
Let   $\{u_k\}_k$ be a sequence of classical positive solutions to \eqref{eq:elliptic:problem} and assume that
\begin{equation}\label{uk:infty}
\lim_{k \to \infty} \|u_k\|_{\infty} = + \infty.
\end{equation}

Let $C,\de>0$ be as in Theorem \ref{th:comp}.
Let $x_k \in \overline{\Om_{\delta}}$ be such that
$$u_k(x_k) = \max_{\Om_{\delta}} u_k.$$
Since
$\, 0<\frac{1}{C}\leq \frac{u_k(x_k)}{\|u_k\|_{\infty}}\leq 1  ,\, $
by taking a subsequence if needed, we may assume that
$\lim_{k\to \infty} \frac{u_k(x_k)}{\|u_k\|_{\infty}}=L>0$
and there exists $x_0 \in \overline{\Om_{\delta}}$ such that
$\lim_{k\to \infty} x_k = x_0\in\overline{\Om_{\delta}}.$

\smallskip

As observed in \cite{Turner,Nussbaum,Brezis-Turner}, \cite[p. 44]{Figueiredo-Lions-Nussbaum}, there exists a constant $C_1>0$ such that
$$\int_{\Om}u_k\phi_1\leq C_1\qquad \int_{\Om}f(u_k)\phi_1\leq C_1.$$

Let $d_0=dist (x_0,\p\Om)\geq\de>0,$
and let $B_{\rho}=B(x_0,\rho)$ for $\rho\in(0,\de)$. Since
$$\min_{B_{d_0/2}}\phi_1\int_{B_{d_0/2}} f(u_k) \leq\int_{B_{d_0/2}} f(u_k)\phi_1 \leq\int_\Om f(u_k)\phi_1 ,$$
there exists a constant $C_2$ independent of $k$ such that
\begin{equation}\label{cota:f:L1}
\int_{B_{d_0/2}} u_k \leq C_2,\qquad \int_{B_{d_0/2}} f(u_k) \leq C_2.
\end{equation}

Let us now  define
\begin{equation}\label{w}
w_k(x) =\frac{ u_k(x)} { \|u_k\|_{\infty}},\qq{for}x\in\Om,
\end{equation}
hence
$w_k(x_k) \geq \frac{1}{C} >0$ for all $k,$ and
\begin{equation}\label{l}
\lim_{k\to \infty} w_k(x_k)=L>0.
\end{equation}
Note also that
$$-\Delta w_k(x)=-\frac{1}{\|u_k\|_{\infty}}\ \Delta u_k(x) = \frac{1}{\|u_k\|_{\infty}}\,f
\big(u_k(x)\big)\qq{for all}x\in\Om.
$$

Let us fix $q\in\left(1,\frac{N+2}{4}\right),$ for $N\geq 3.$ We observe that $N^\star (1-1/q)<1.$ Taking into account  hypothesis {\rm (H2)} on $f,$  \eqref{cota:f:L1}, and \eqref{uk:infty} we deduce
\begin{eqnarray}\label{cota:f:Lq}
\frac{1}{\|u_k\|_{\infty}}\,\left(\int_{B_{d_0/2}} \left|f
\big(u_k(x)\big)\right|^q\right)^{1/q} &=&\frac{1}{\|u_k\|_{\infty}}\,\left(\int_{B_{d_0/2}} \left|f
\big(u_k(x)\big)\right|^{q-1}\left|f
\big(u_k(x)\big)\right|\right)^{1/q}\nonumber\\
&\leq & C \|u_k\|_{\infty}^{N^\star (1-1/q)-1}\to 0\qq{as}k\to\infty.
\end{eqnarray}
From interior elliptic regularity results, see \cite{ADN1,ADN2}, we can write
\begin{equation}\label{cota:v:W2q}
\|w_k\|_{W^{2,q} (B_{d_0/4})}\leq C\left( \|u_k\|_{\infty}^{N^\star (1-1/q)-1}+\|w_k\|_{L^q (B_{d_0/2})}\right)\leq C.
\end{equation}

From compact Sobolev imbeddings, at least for a subsequence,
\begin{equation}\label{converg:w}
w_k\to w\qq{as} k\to\infty ,\qquad\mbox{in}\quad W^{1,q} (B_{d_0/4}).
\end{equation}

Due to $w_k> 0$ then $w\geq 0$ in $B_{d_0/4}.$ Moreover, either $\int_{B_{d_0/4}}w>0$ or $\int_{B_{d_0/4}}w=0.$
Assume that $\int_{B_{d_0/4}}w=C>0.$ From \eqref{converg:w} and compact imbeddings, we obtain
$\int_{B_{d_0/4}}w_k\to \int_{B_{d_0/4}}w,$ as $k\to\infty,$
therefore
$\int_{B_{d_0/4}}w_k\geq\frac{C}{2}$ for any $k$ big enough.
By definition
$\int_{B_{d_0/4}}w_k=\frac{1}{\|u_k\|_{\infty}}\,
\int_{B_{d_0/4}}u_k,$
therefore
$\int_{B_{d_0/4}}u_k\geq\frac{C}{2}\|u_k\|_{\infty}\to\infty, $ as $k\to\infty,$
which contradicts \eqref{cota:f:L1}.
Consequently $\int_{B_{d_0/4}}w=0,$ and therefore
$\int_{B_{d_0/4}}w_k\to 0,$ as $k\to\infty.$
By definition of $w_k$, see \eqref{w}, $\|w_k\|_{L^{\infty}}\leq 1$, which implies
$$0\leq \int_{B_{d_0/4}}w_k^q\leq \|w_k\|_{L^{\infty} (B_{d_0/4})}^{q-1}\int_{B_{d_0/4}}w_k\leq \int_{B_{d_0/4}}w_k\to 0\qq{as}k\to\infty,
$$
therefore, $\|w_k\|_{L^{q} (B_{d_0/4})}\to 0.$  Plugging this in \eqref{cota:v:W2q}, we deduce
$\|w_k\|_{W^{2,q} (B_{d_0/8})}\to 0$ as $k\to\infty.$
Due to \eqref{converg:w}, in particular
$
\|\nabla w\|_{L^{q} (B_{d_0/8})}=0,
$
and from Hölders inequality
$\|\nabla w\|_{L^{1} (B_{d_0/8})}=0.$

Obviously
$L\leq |L-w(x)|+|w(x)|,$
and integrating on $B_{d_0/16}$ we obtain
\begin{equation}\label{ineq:v:geq:0}
\int_{B_{d_0/16}}|w(x)-L|\geq L|B_{d_0/16}|-\int_{B_{d_0/16}}w=L|B_{d_0/16}|>0.
\end{equation}

On the other hand, adding $\pm w(x_k+y),$ $\pm w_k(x_k+y),$ $\pm w_k(x_k),$ for $y\in B_{d_0/16}(0),$  we have
$$
\int_{B_{d_0/16}}|w(x)-L|=\int_{B_{d_0/16}(0)}|w(x_0+y)-L|\leq I_1+I_2+I_3+I_4,$$
where
\begin{eqnarray*}
  I_1 &=& \int_{B_{d_0/16}(0)}|w(x_0+y)-w(x_k+y)|, \\
  I_2 &=& \int_{B_{d_0/16}(0)}|w(x_k+y)-w_k(x_k+y)|, \\
  I_3 &=& \int_{B_{d_0/16}(0)}|w_k(x_k+y)-w_k(x_k)|, \\
  I_4 &=& \int_{B_{d_0/16}(0)}|w_k(x_k)-L|.
\end{eqnarray*}
Set $g(t)=w(tx_0+(1-t)x_k+y),$ then $g'(t)=(x_0-x_k)\cdot\nabla w(tx_0+(1-t)x_k+y),$ and thus
$$w(x_0+y)-w(x_k+y)=g(1)-g(0)=\int_0^1g'(t)\,dt=\int_0^1
(x_0-x_k)\cdot\nabla w(tx_0+(1-t)x_k+y)\,dt.
$$
Therefore
$$|w(x_0+y)-w(x_k+y)|\leq |x_0-x_k|\int_0^1|\nabla w(tx_0+(1-t)x_k+y)|\,dt,
$$
and consequently, integrating on $B_{d_0/16}(0)$ and using Fubini's theorem we deduce
\begin{eqnarray*}
    I_1 &\leq & |x_0-x_k| \int_{B_{d_0/16}(0)}\left(\int_0^1|\nabla w(tx_0+(1-t)x_k+y)|\,dt \right)\,dx\\
     &=& |x_0-x_k|\int_0^1\left(\int_{B_{d_0/16}(0)} |\nabla w(tx_0+(1-t)x_k+y)|\,dx\right)\,dt \\
     &\leq &  |x_0-x_k|\, \|\nabla w\|_{L^{1} (B_{d_0/8})}=0.
  \end{eqnarray*}

Moreover, due to \eqref{converg:w}
$$
I_2\leq \int_{B_{d_0/8}} |w-w_k|\to 0\qq{as}k\to\infty.
$$
Reasoning as we did to bound $I_1,$  we can write
$$
I_3\leq \int_{B_{d_0/16}(0)}|y||\nabla w_k(x_k+ty|\leq \frac{d_0}{8} \int_{B_{d_0/8}}|\nabla w_k|\to 0\qq{as}k\to\infty.
$$
Finally
$$
I_4= |w_k(x_k)-L|\int_{B_{d_0/16}(0)}\, dx= |B_{d_0/16}||w_k(x_k)-L|\to 0\qq{as}k\to\infty.
$$
Therefore, $\int_{B_{d_0/16}(0)}|w(x_0+y)-L|=0,$
which contradicts \eqref{ineq:v:geq:0} and completes the proof.
\end{proof}

\appendix
\section{}
In this Appendix we prove that  for any boundary point of a $C^2$ domain, the maximal cap in the transformed domain is nonempty. This could seem surprising in presence of highly oscillatory boundaries.
For example, suppose  the boundary of $\Omega$ includes $ \G_2 =\big\{(x,f(x)):f(x):= x^5\sin \left(\frac{1}{x}\right),\ x\in [-0.01,0.01]\big\},$
see fig. \ref{fig6_7_8_9}(b). Let $ h(\G_2)$ be the image through the inversion map into the unit ball $B,$ and let $\G_3$ be  the arc of the boundary  $\p B$ given by $\G_3=\{(x,g(x)) : g(x):=\sqrt{1-x^2},\ x\in [-0.01,0.01]\},$ see fig. \ref{fig6_7_8_9}(c). At this scale, the oscillations are not appreciable. We plot in \ref{fig6_7_8_9}(d) the derivative of the 'vertical' distance between the boundary $ \G_2$ and the ball, concretely we plot $f'(x)-g'(x)$ for $x\in [-0.01,0.01]$.
We plot in \ref{fig6_7_8_9}(e) the second derivative of the 'vertical' distance between the boundary and the ball, which is $f''(x)-g''(x)$ for $x\in [-5\cdot 10^{-4},5\cdot 10^{-4}]$.
Let us observe that this second derivative is strictly positive,  and that $f''(0)-g''(0)=1.$
Consequently, the first derivative is strictly increasing,
and therefore the 'vertical' distance $f(x)-g(x)$ does not oscillate.

Moreover, let us consider the image through the inversion map of the straight line $y=1,$ i.e.
$h(x,1) =h\left(\{(x,1),\ x\in [-0.01,0.01]\}\right) .$
In fig. \ref{fig6_7_8_9}(f)-(g) we plot the second coordinate of the difference $ h(\G_2)-h(x,1).$
The oscillation phenomena is present here.
In fig. \ref{fig6_7_8_9}(h) we plot the second coordinate of the difference $ h(\G_2)-h(\p B).$ This difference does not oscillate.

\begin{figure}[htb]
$$\kern -2em
\begin{array}{ccc}
\includegraphics[width=4cm]{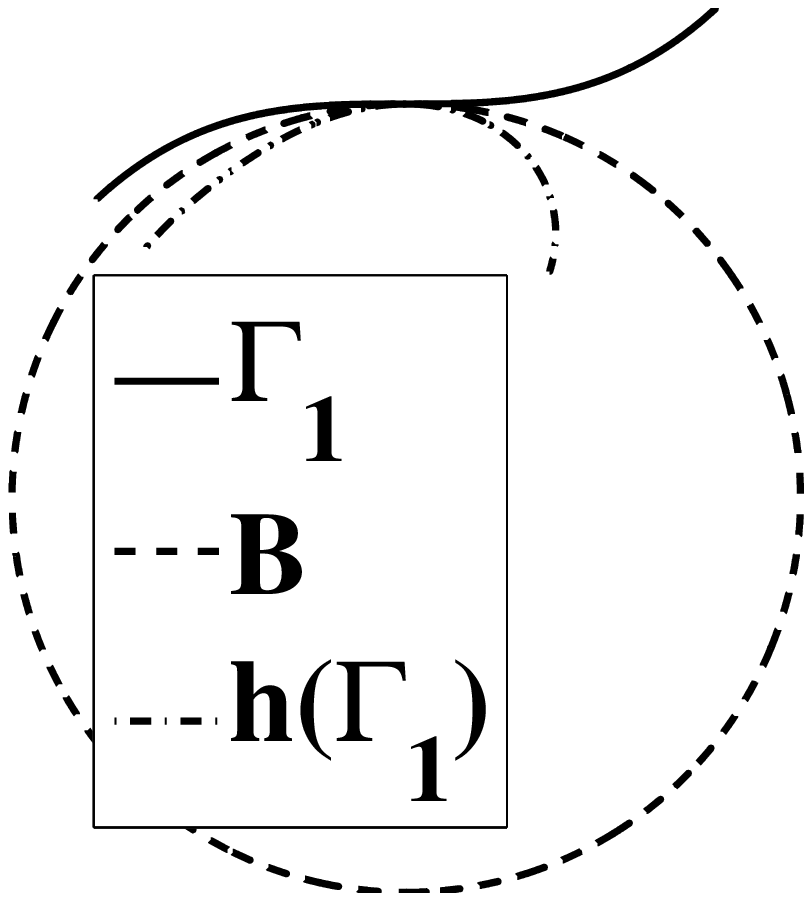}&
\includegraphics[width=4cm]{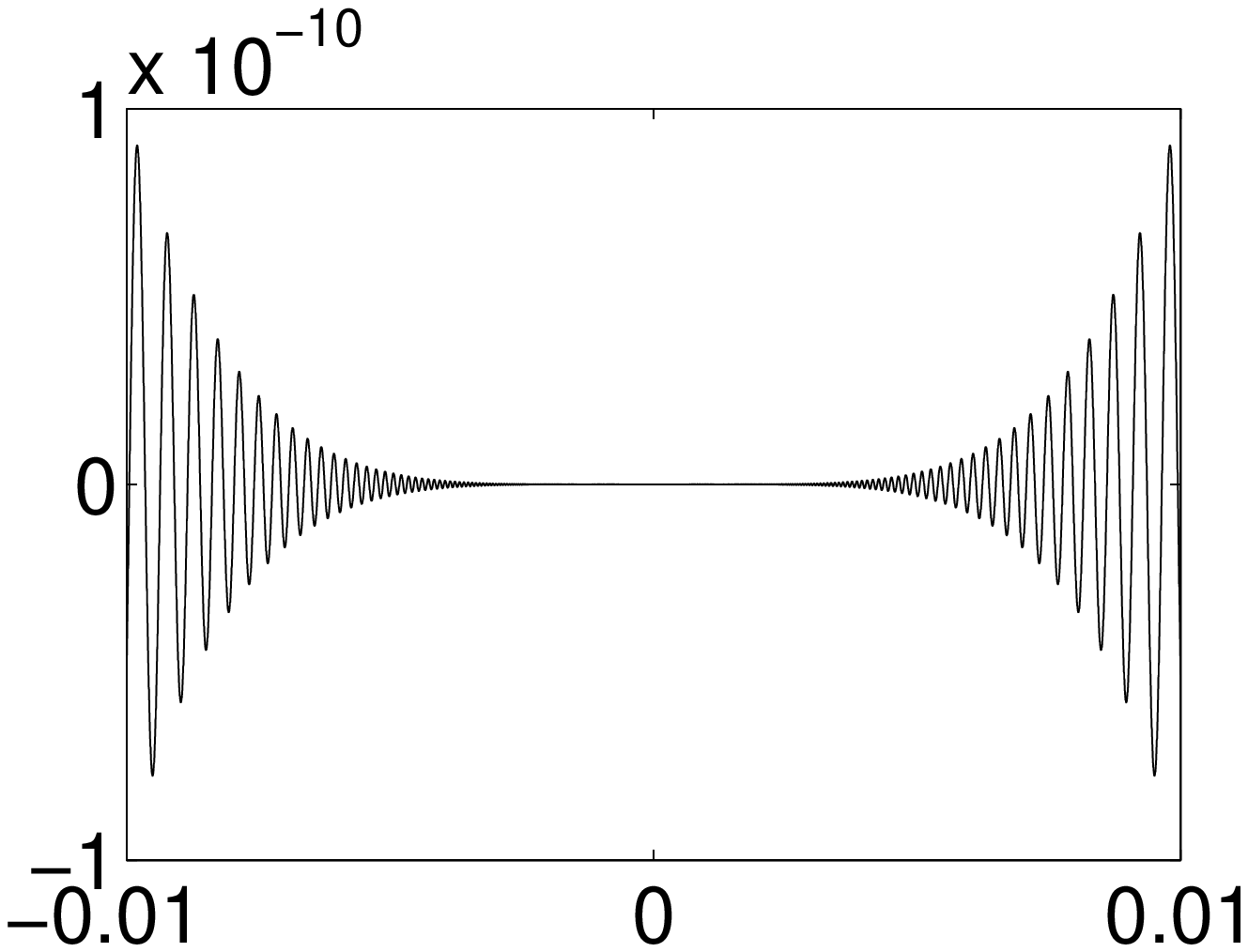} &
\includegraphics[width=4cm]{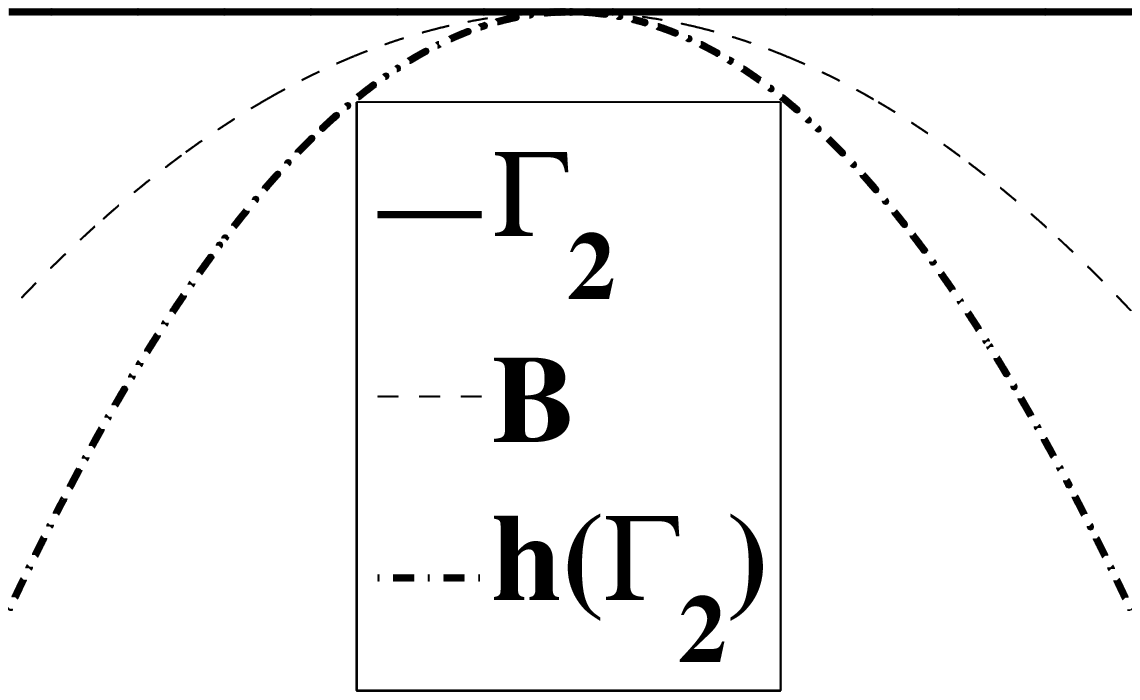} \\
 {\rm (a)} & {\rm (b)} & {\rm (c)}
 \end{array}
$$\kern -2em
$$\begin{array}{cc}
\includegraphics[width=4cm]{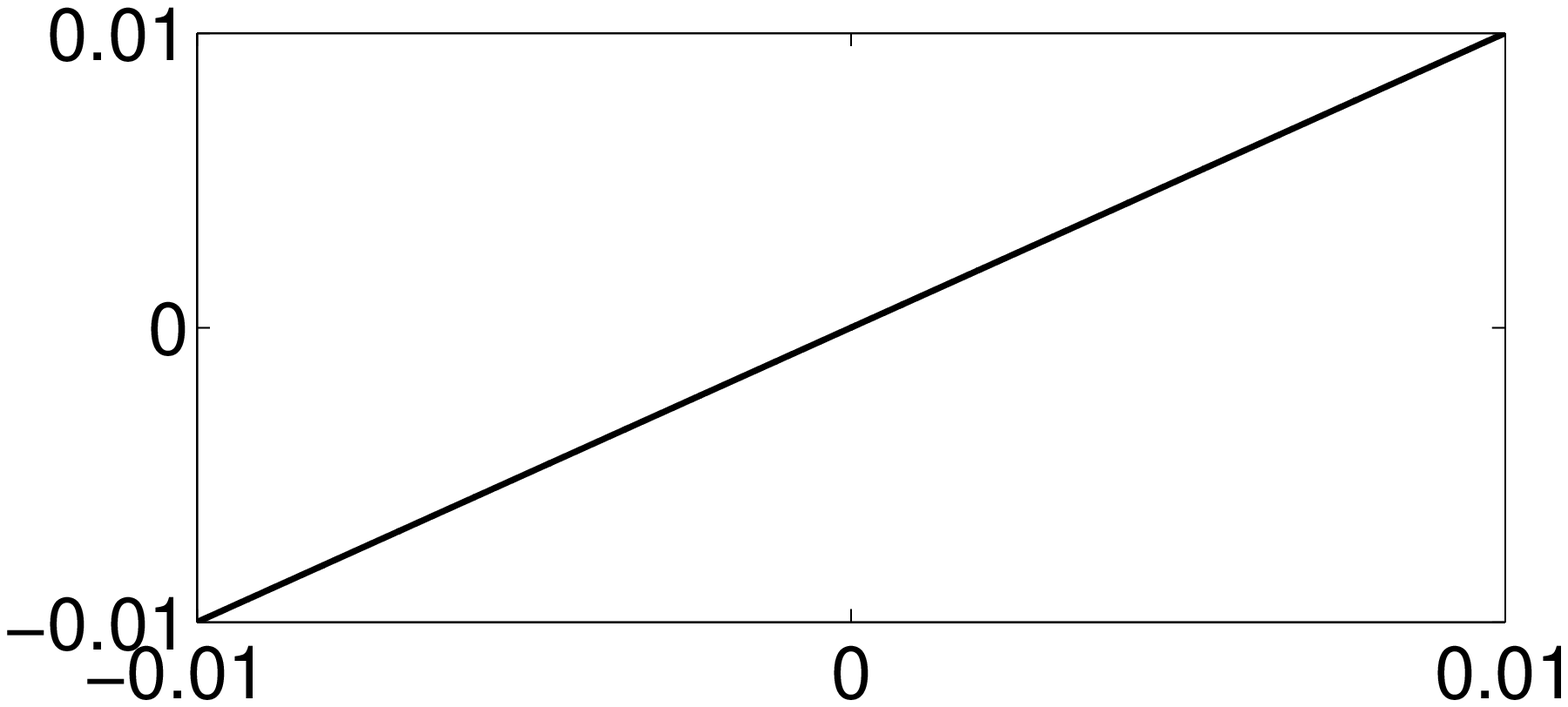} &
\includegraphics[width=4cm]{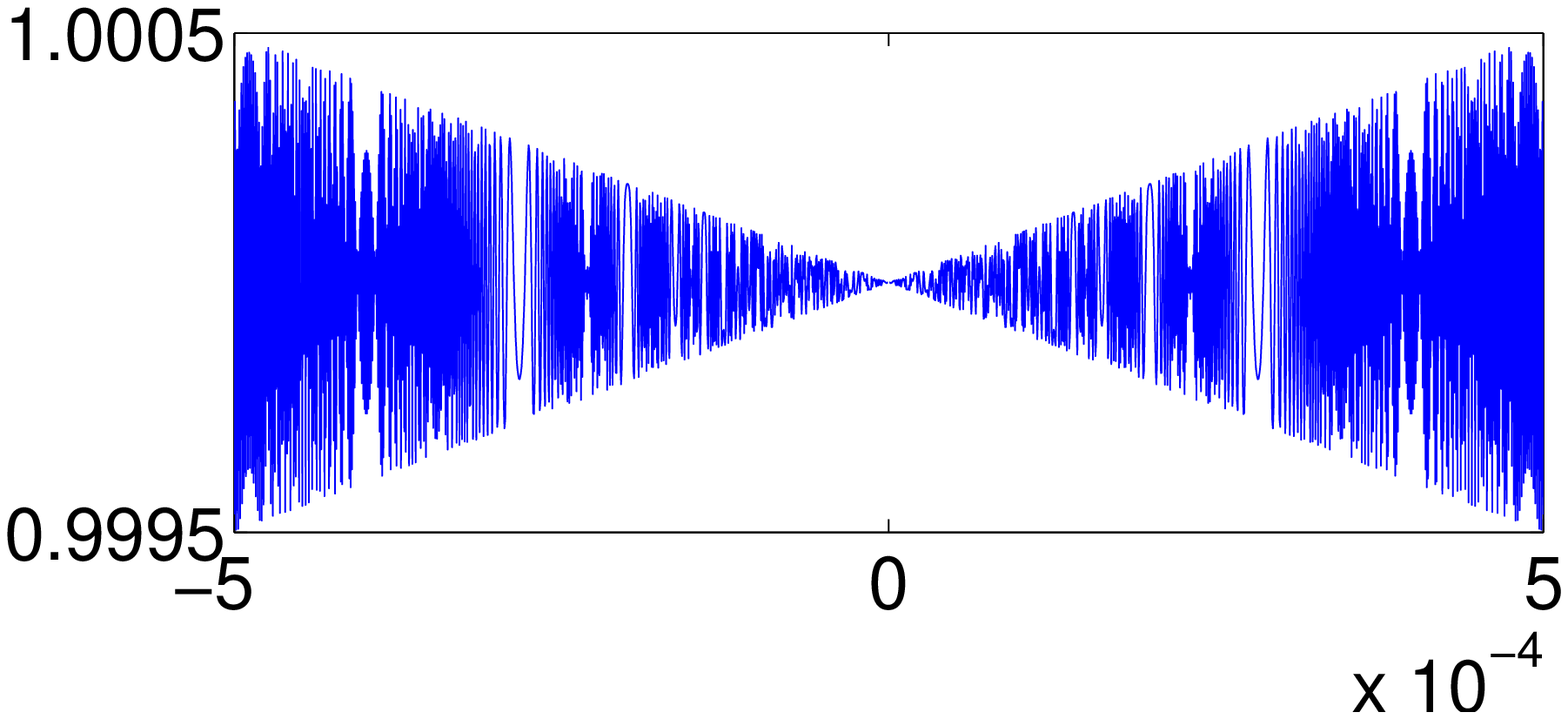}\\
 {\rm (d)} & {\rm (e)}
\end{array}$$
$$\kern -2em
\begin{array}{ccc}
\includegraphics[width=4cm]{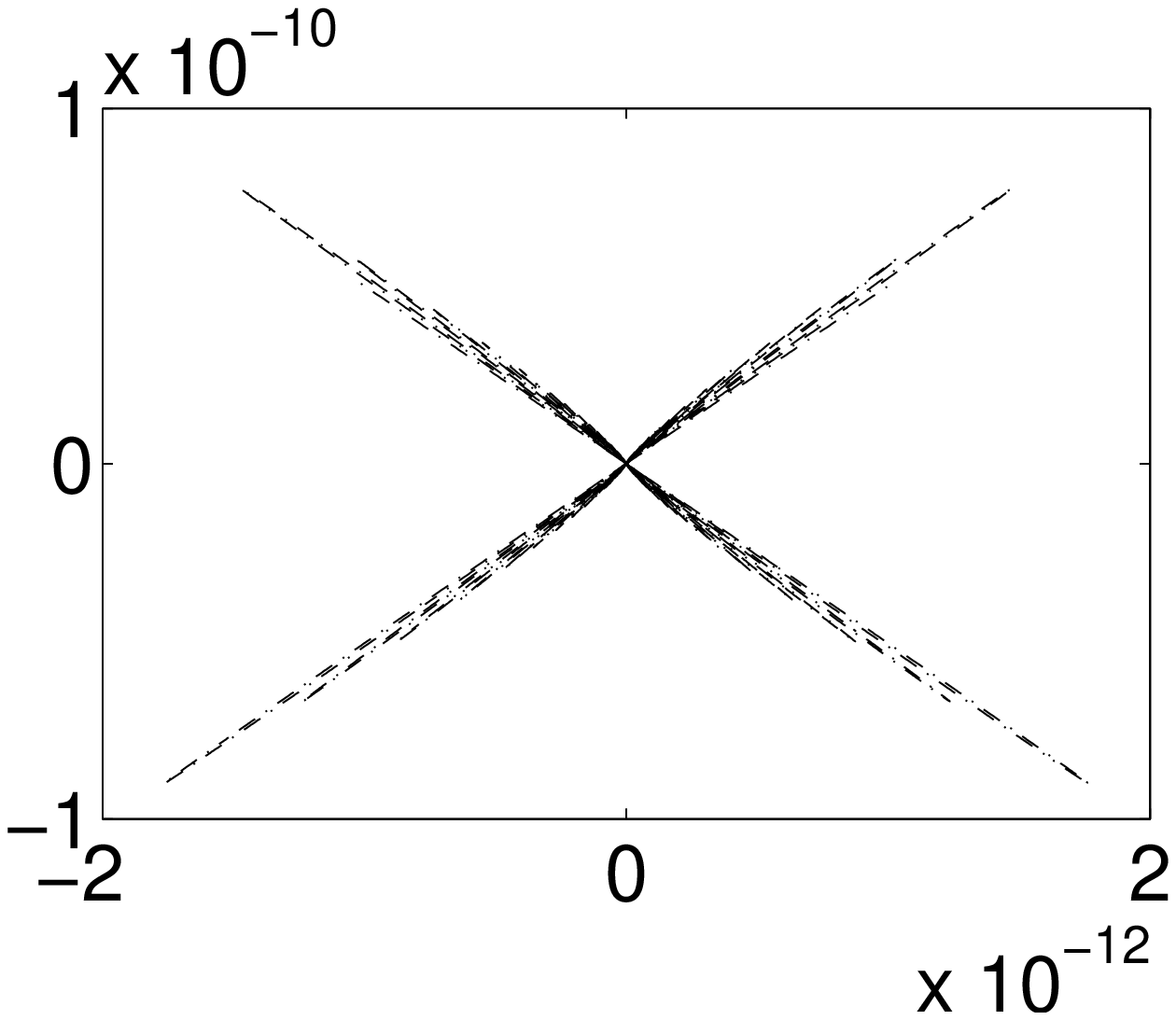} & \includegraphics[width=4cm]{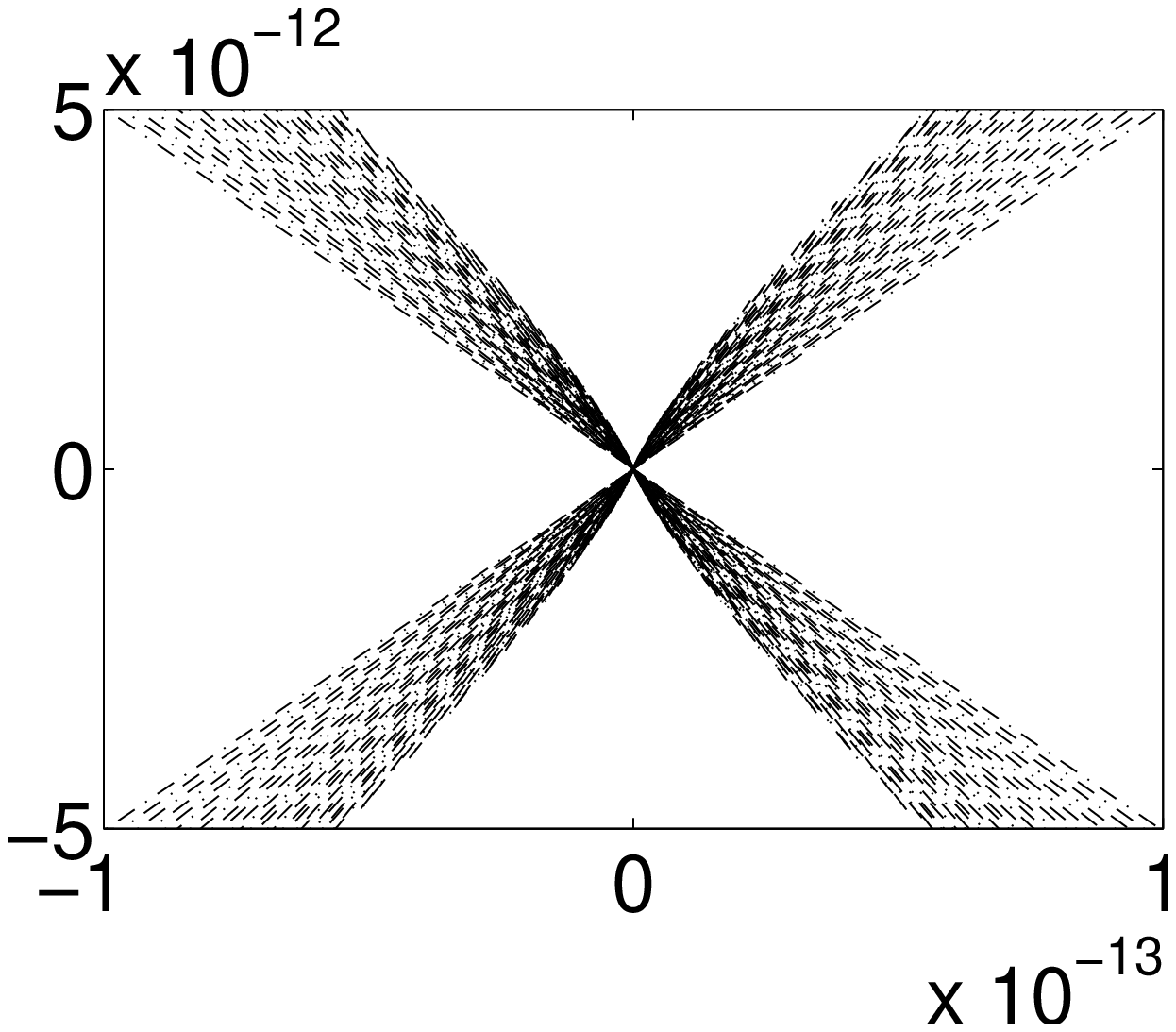} &
\includegraphics[width=4cm]{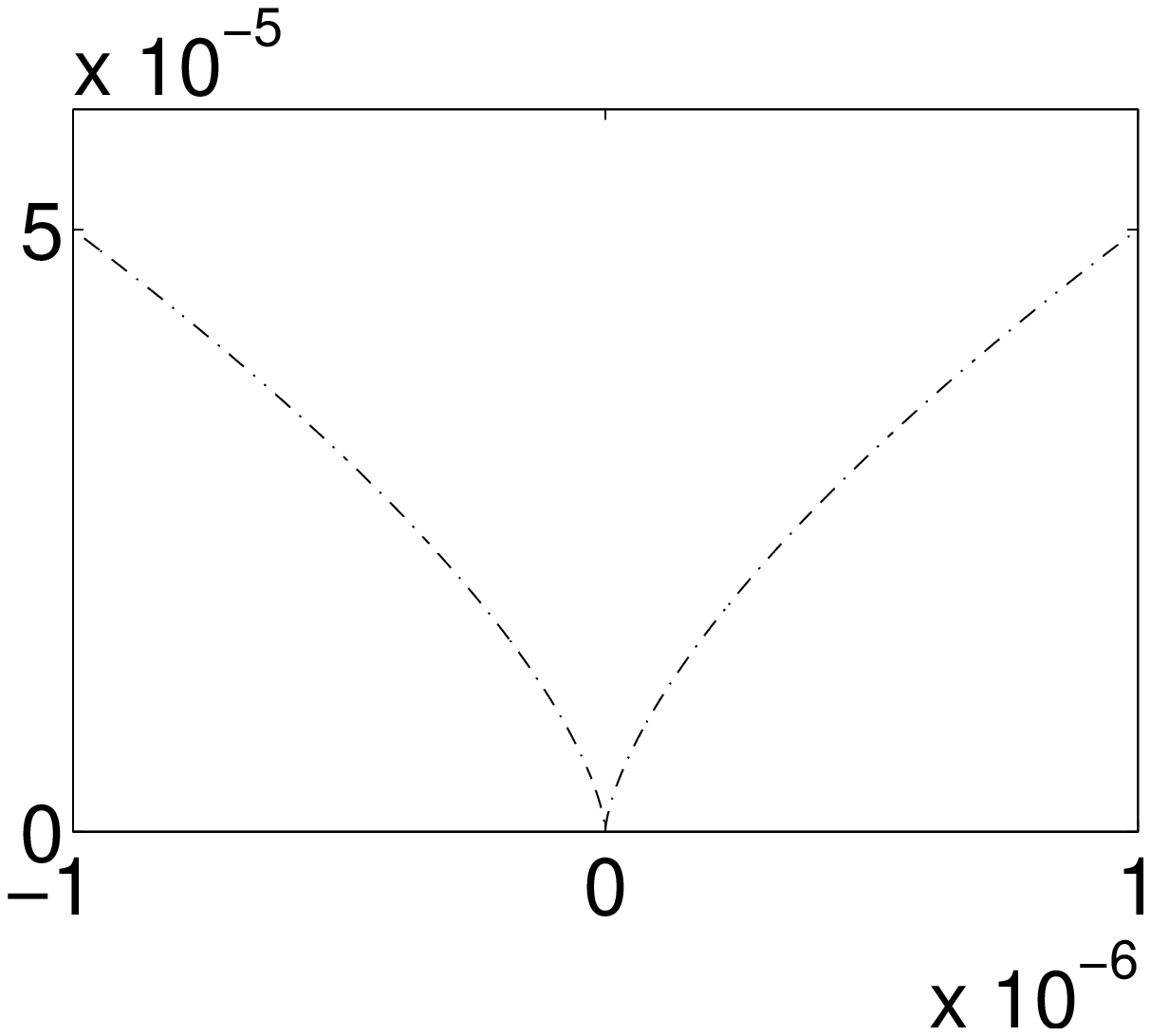} \\
 {\rm (f)} & {\rm (g)} & {\rm (h)}
\end{array}
$$
\caption{(a) An inflection point at the boundary $\G_1 $ joint with the inversion $h(\G ),$ and the unit circumference;
(b) A degenerated critical point at the boundary $ \G_2 $;
(c) $\G_2 $ joint with its inversion  into the unit ball, $h(\G_2),$ and the arc of circumference, $\G_3$;
(d) $f'(x)-g'(x)$ for $x\in [-0.01,0.01]$;
(e) $f''(x)-g''(x)$ for $x\in [-5\cdot 10^{-4},5\cdot 10^{-4}]$;
(f) Second coordinate of the difference  $h(\G_2 )-h(x,1)$ where $ h(x,1)$ is the image of the straight line $y=1$;
(g) a zoom of the same graphic;
(h) Second coordinate of the  difference $ h(\G_2 )-h(\G_3).$}
\label{fig6_7_8_9}
\end{figure}

\smallskip

\smallskip

In fig. \ref{fig6_7_8_9}(a) we draw the inversion of the boundary into the unit ball at an inflexion point; more
precisely we set $\G_1: =\big\{(x,f(x)):f(x)= \frac{x^3}{2}+1,\ x\in [-\pi/4,\pi/4]\big\},$
which has an inflexion point at   $x=0.$

\smallskip

Let $h$ denote the  inversion map defined in \eqref{def:h}, and let $\widetilde{\Om}=h(\Om)$ denote the image through the inversion map into the  ball $B$. For any $x_0\in\p\Om$, let $\tilde{n}_i(x_0)$ be the normal inward at $x_0$ in the transformed domain $\tilde{\Om},$
and let $\widetilde{\Sigma}= \widetilde{\Sigma} (\tilde{n}_i(x_0))$ be its maximal cap,  see fig. \ref{fig3_4_5}(b).

\begin{lem}\label{lem:convex}
If   $\Omega \subset \R ^N $ is a bounded  domain  with $C^{2}$  boundary, then for
any $x_0\in\p\Om,$ there exists a maximal cap $\widetilde{\Sigma}=\widetilde{\Sigma} (\tilde{n}_i(x_0))$  non empty.
\end{lem}
\begin{proof}
For convenience,  we assume $x_0=(0,\cdots,0,1),$  and $B$ is the unit ball with center at the origin such that $\p B\cap\p\Om=x_0.$ Let $\{(x', \psi(x'));
\|x\| < a\}$, $a>0$, denote a parametrization of $\partial \Omega$ in a neighborhood of $x_0$.  Hence
\begin{equation}\label{eq:der:psi:critical:point}
\psi (0')=1,\qq{and}  \nabla_{N-1}\psi(0')=0'.
\end{equation}

Let $h(\Om)$ stand for the image through the inversion map into the unit ball.
From definition,  $h(\p\Om\cap B(x_0))$ is given by
\begin{equation}\label{eq:reflected:bd}
h\Big(x',\psi(x')\Big)=\frac{ \left(x',\psi(x')\right)}{|x'|^2+\psi(x')^2}, \qq{for} x'\in{\mathcal N}.
\end{equation}
Set $y=h(x',\psi(x'))$ for $x'\in{\mathcal N}$ and with $y=(y',y_N).$  Since
$$
y'=\frac{x'}{|x'|^2+\psi(x')^2},\quad y_N=\frac{\psi(x')}{|x'|^2+\psi(x')^2},
\quad\mbox{and }
|y'|^2+y_N^2=\frac{1}{|x'|^2+\psi(x')^2},
$$
for $x'\in{\mathcal N},$ then
$x'=\frac{y'}{|y'|^2+y_N^2},$ for $y'\in{\mathcal N'},$
where $y'\in{\mathcal N'}$ if and only if $
y'=\frac{x'}{|x'|^2+\psi(x')^2}$ for some $x'\in{\mathcal N}.
$
Therefore
\begin{small}
$$ y_N=\frac{\psi\left(\frac{y'}{|y'|^2+y_N^2}\right)}{
|y'|^2+\psi\left(\frac{y'}{|y'|^2+y_N^2}\right)^2}, \qq{for} y'\in{\mathcal N'},
$$
\end{small}
and
$$h(\p\Om\cap B(x_0))=\left\{(y',y_N)\in\mathbb{R}^{N-1}\times\mathbb{R}: \ F(y',y_N)=0,\ y'\in{\mathcal N'}\right\},$$
where
\begin{equation}\label{def:F}
F(y',y_N):=y_N \left[|y'|^2+\psi\left(\frac{y'}{|y'|^2+y_N^2}\right)^2\right]-\psi\left(\frac{y'}{|y'|^2+y_N^2}\right).
\end{equation}


Differentiating \eqref{def:F} with respect to $y_N$  we obtain
\begin{eqnarray*}
  \frac{\p F}{\p y_N}(y',y_N) &=& \left[|y'|^2+\psi\left(\frac{y'}{|y'|^2+y_N^2}\right)^2\right]+y_N\frac{\p }{\p y_N}\left[|y'|^2+\psi\left(\frac{y'}{|y'|^2+y_N^2}\right)^2\right]\\
  && -\sum_{i=1}^{N-1}\frac{\p \psi}{\p y_i}\left(\frac{y'}{|y'|^2+y_N^2}\right)\frac{\p }{\p y_N}\left(\frac{y_i}{|y'|^2+y_N^2}\right).
\end{eqnarray*}
Substituting at $(y',y_N)=(0',1)$ and taking into account \eqref{eq:der:psi:critical:point}
\begin{small}
\begin{eqnarray*}
\frac{\p F}{\p y_N} (0',1)   &=& \left.  1+2\psi\left(\frac{y'}{|y'|^2+y_N^2}\right)\sum_{i=1}^{N-1}\frac{\p \psi}{\p y_i}\left(\frac{y'}{|y'|^2+y_N^2}\right)\frac{\p }{\p y_N}\left(\frac{ y_i}{|y'|^2+y_N^2}\right)\right|_{(y',y_N)=(0',1)}\\
    &=& 1\neq 0.
\end{eqnarray*}
\end{small}
Therefore, by the Implicit Function Theorem  there exists an open neighborhood of $0'$, $B_\de (0')\subset \mathbb{R}^{N-1}$ and a unique function $\phi : B_\de (0')\to\R,$ $\phi\in C^2(B_\de (0')),$ such that $\phi(0')=1,$ and
\begin{equation}\label{eq:F}
F(y',\phi (y'))=0\qq{for all}y'\in B_\de (0').
\end{equation}

Differentiating  \eqref{eq:F} with respect to $y_j,$ $j=1,\cdots, N-1,$ using the chain rule and substituting at the point $(0',1),$ we obtain
\begin{equation}\label{eq:der:phi}
\frac{\p F}{\p y_j}(0',1)+\frac{\p F}{\p y_N}(0',1)\frac{\p \phi}{\p y_j}(0')=0,\qq{for}j=1,\cdots N-1.
\end{equation}

On the other hand, differentiating  \eqref{def:F} with respect to $y_j$ and using the  chain rule we obtain
\begin{small}
\begin{equation*}
  \frac{\p F}{\p y_j}(y',y_N) = y_N\frac{\p }{\p y_j}\left[|y'|^2+\psi\left(\frac{y'}{|y'|^2+y_N^2}\right)^2\right] -\sum_{i=1}^{N-1}\frac{\p \psi}{\p y_i}\left(\frac{y'}{|y'|^2+y_N^2}\right)\frac{\p }{\p y_j}\left(\frac{y_i}{|y'|^2+y_N^2}\right).
\end{equation*}
\end{small}
Substituting at $(y',y_N)=(0',1)$ and taking into account \eqref{eq:der:psi:critical:point}
\begin{small}
\begin{equation*}
\frac{\p F}{\p y_j} (0',1)   = \left.  2\psi\left(\frac{y'}{|y'|^2+y_N^2}\right)\sum_{i=1}^{N-1}\frac{\p \psi}{\p y_i}\left(\frac{y'}{|y'|^2+y_N^2}\right)\frac{\p }{\p y_j}\left(\frac{ y_i}{|y'|^2+y_N^2}\right)\right|_{(y',y_N)=(0',1)}=0.
\end{equation*}
\end{small}
Consequently, by \eqref{eq:der:phi}
\begin{equation}\label{eq:der:phi:critical:point}
\nabla_{N-1} \phi(0')=0'.
\end{equation}

Let us define
$$ g(y'):=\psi\left(\frac{y'}{|y'|^2+\phi (y')^2}\right),
\quad\mbox{and }
 G(y'):=\frac{g(y')}{
|y'|^2+g(y')^2}, \qquad\mbox{for } y'\in B_\de (0').
$$
By  \eqref{eq:der:psi:critical:point},
$g(0')=1,$ and $G(0')=1.$ Moreover,
$$\left\{(y',y_N)\in\mathbb{R}^{N-1}\times\mathbb{R}: \ y_N=G(y'),\ y'\in B_\de (0')\right\}\subset h(\p\Om)\cap B(x_0),$$
and
$$\left\{(y',y_N)\in\mathbb{R}^{N-1} \times\mathbb{R}: \ y_N<G(y'),\ y'\in B_\de (0')\right\}\subset h(\Om)\cap B(x_0).$$
Let us see that   there exists   $0<\de'\leq \de$ such that
$$U:=\left\{(y',y_N)\in\mathbb{R}^{N-1} \times\mathbb{R}: \ y_N<G(y'),\ y'\in B_{\de'} (0')\right\},$$
is a convex set. To achieve this, we   use a characterization of convexity in the twice continuously differentiable case, see \cite[p. 87-88]{fenchel}.
{\it The set $U$ is a convex set if and only if $D^2G(y')$ is negative semidefinite for all $y'\in B_\de (0')$}.
In fact, we will prove that $D^2G(0')$ is negative definite and by continuity, there exists some $\de'>0$ such that $D^2G(y')$ is negative semidefinite for all $y'\in B_{\de'} (0').$
Differentiating
$$\frac{\p g}{\p y_j}=\sum_{i=1}^{N-1}\frac{\p \psi}{\p y_i}\left(\frac{y'}{|y'|^2+\phi (y')^2}\right)\frac{\p }{\p y_j}\left(\frac{y_i}{|y'|^2+\phi (y')^2}\right),$$
and
$$\frac{\p G}{\p y_j}=\frac{ \p_j g}{|y'|^2+g(y')^2}
- \frac{ 2g(y')\left(y_j+g\D\p_j g\right)}{\D\left(|y'|^2+g(y')^2\right)^2},\qq{for}j=1,\cdots N-1,$$
where $\p_j g =\frac{\p g}{\p y_j}.$  Substituting at $y'=0',$ and taking into account \eqref{eq:der:psi:critical:point} we deduce
\begin{equation}\label{eq:der:g:critical:point}
\nabla_{N-1} g(0')=0',
\qq{and}
\nabla_{N-1} G(0')=0'.
\end{equation}

Taking second derivatives for  $k=1,\cdots N-1,$
we obtain
\begin{eqnarray*}
  \frac{\p^2 g}{\p y_k\p y_j} &=& \sum_{i=1}^{N-1}\frac{\p}{\p y_k}\left[\frac{\p \psi}{\p y_i}\left(\frac{y'}{|y'|^2+\phi (y')^2}\right)\right]\frac{\p }{\p y_j}\left(\frac{ y_i}{|y'|^2+\phi (y')^2}\right)\\
  && +\sum_{i=1}^{N-1}\frac{\p \psi}{\p y_i}\left(\frac{y'}{|y'|^2+\phi (y')^2}\right)\frac{\p^2 }{\p y_k\p y_j}\left(\frac{y_i}{|y'|^2+\phi (y')^2}\right),
\end{eqnarray*}
and
\begin{eqnarray*}
  \frac{\p^2 G}{\p y_k\p y_j} &=& \frac{ \p^2_{kj} g}{|y'|^2+g(y')^2}
- \frac{ 2\p_j g(y')\left(y_k+g\D\p_k g\right)}{\D\left(|y'|^2+g(y')^2\right)^2}\\
  && - \frac{ 2\p_k g(y')\left(y_j+g\D\p_j g\right)+2 g(y')\p_k\left(y_j+g\D\p_j g\right)}{\D\left(|y'|^2+g(y')^2\right)^2}\\
  && +\frac{ 4g(y')\left(y_j+g\D\p_j g\right)\left(y_k+g\D\p_k g\right)}{\D\left(|y'|^2+g(y')^2\right)^3},
\end{eqnarray*}
where $\p^2_{kj}  =\frac{\p^2}{\p y_k\p y_j}.$
Substituting at $y'=0'$ and taking into account \eqref{eq:der:psi:critical:point} we deduce
\begin{equation*}
  \frac{\p^2 g}{\p y_k\p y_j}(0') = \left.\sum_{i=1}^{N-1}\frac{\p}{\p y_k}\left[\frac{\p \psi}{\p y_i}\left(\frac{y'}{|y'|^2+\phi (y')^2}\right)\right]\frac{\p }{\p y_j}\left(\frac{y_i}{|y'|^2+\phi (y')^2}\right)\right|_{y'=0'}.
\end{equation*}
Substituting at $y'=0'$ and taking into account \eqref{eq:der:g:critical:point} we deduce
\begin{eqnarray*}
\frac{\p^2 G}{\p y_k\p y_j}(0') &=& \left.\frac{ \p^2_{kj} g}{|y'|^2+g(y')^2}
 - \frac{ 2 g(y')\p_k\left(y_j+g\D\p_j g\right)}{\D\left(|y'|^2+g(y')^2\right)^2} \right|_{y'=0'}\\
 &=& \p^2_{kj} g(0')-2(\de_{jk}+\p^2_{kj} g(0'))=-2\de_{jk}-\p^2_{kj} g(0').
\end{eqnarray*}
Due to
\begin{equation*}
  \frac{\p }{\p y_j}\left(\frac{y_i}{|y'|^2+\phi (y')^2}\right) = \frac{\de_{ij}}{|y'|^2+\phi (y')^2}-\frac{2y_i
  \left(y_j+\phi\D\p_j\phi\right)}{\D\left(|y'|^2+g(y')^2\right)^2},
\end{equation*}
where $\de_{ij}$ is the Kronecker's delta, substituting at $y'=0'$ and taking into account \eqref{eq:der:phi:critical:point} we can write
\begin{equation}\label{eq:der:y:critical:point}
\left.\frac{\p }{\p y_j}\left(\frac{y_i}{|y'|^2+\phi (y')^2}\right)\right|_{y'=0'} = \de_{ij}.
\end{equation}
Moreover,
\begin{equation*}
   \frac{\p}{\p y_k}\left[\frac{\p \psi}{\p y_i}\left(\frac{y'}{|y'|^2+\phi (y')^2}\right)\right]=
   \sum_{m=1}^{N-1}\frac{\p^2 \psi}{\p y_m\p y_i}\left(\frac{y'}{|y'|^2+y_N^2}\right)\frac{\p }{\p y_k}\left(\frac{y_m}{|y'|^2+y_N^2}\right),
\end{equation*}
substituting at $y'=0'$ and taking into account \eqref{eq:der:y:critical:point} we can write
\begin{equation*}
\left.\frac{\p}{\p y_k}\left[\frac{\p \psi}{\p y_i}\left(\frac{y'}{|y'|^2+\phi (y')^2}\right)\right]\right|_{y'=0'}=
   \frac{\p^2 \psi}{\p y_k\p y_i}(0'),
\end{equation*}
Let $A:=\left(\p^2_{kj}\psi(0')\right)_{j,k=1,\cdots N-1},$ then
\begin{equation}\label{hess:g:G}
    \left(\p^2_{kj} g(0')\right)_{j,k=1,\cdots N-1} = A,
    \quad \mbox{and }
    \left(\p^2_{kj} G(0')\right)_{j,k=1,\cdots N-1} =-(2I_{N-1} +A),
\end{equation}
where $I_{N-1}$ is the identity matrix.

From hypothesis $\p B\cap\p\Om=x_0.$ Therefore the 'vertical' distance (distance in the $x_N$ coordinate) between $\p\Om$ and $\p B$ is strictly positive
i.e.
$$\psi(x')>\sqrt{1-|x'|^2}\qq{for all} x'\in{\mathcal N}\setminus 0'\qq{with} x=(x',x_N)\in\Om\cap B(x_0),$$
or equivalently
$$\left[\psi(x')\right]^2+|x'|^2 >1\qq{for all} x'\in{\mathcal N}\setminus 0'\qq{with} x=(x',x_N)\in\Om\cap B(x_0).$$
Set $H(x'):=\left[\psi(x')\right]^2+|x'|^2 $ for  $x'\in{\mathcal N}$ with $x=(x',x_N)\in\Om\cap B(x_0).$
Then  $H(0')=1$ and from the above inequality, the point $x'=0'$ is an strict minimum of the function $H$.
Due to \eqref{eq:der:psi:critical:point} every derivative of $H$ evaluated at $0'$ is zero, and necessarily  the Hessian matrix of $H$ must be semi positive definite, i.e.
\begin{equation}\label{semi+}
\left.\Big(\p_k\psi \p_j\psi+ \psi \p^2_{kj}\psi+\delta_{kj}\Big)_{j,k=1,\cdots N-1}\right|_{x'=0'} = A+I_{N-1},
\end{equation}
is a semi positive definite matrix.
Hence the matrix $-(A+2I_{N-1})$ is negative definite, and $y'=0'$ is a strict maximum of the function $G.$ As a consequence, there exists a $\de'>0$ such that the matrix $\left(\p^2_{kj}G(y') \right)_{j,k=1,\cdots N-1}$ is  negative definite for all $y'\in B_{\de'} (0').$ Consequently, the set $U$ is a convex set.

Le us now choose $\g =\max\{ G(y')\ |\ y'\in \p B_{\de'} (0')\}.$
Due to $y'=0'$ is a strict maximum of the function $G,$ and that $G(0')=1,$ then $\g<1.$ The cap $\widetilde{\Sigma}_{(1-\g)/2} (-e_N)$
and its reflection $\widetilde{\Sigma}'_{(1-\g)/2}(-e_N)$ are non empty sets contained in $h(\Om).$ Hence the maximal cap $\widetilde{\Sigma}$ contains $ \widetilde{\Sigma}_{(1-\g)/2} (-e_N),$ which in nonempty, which concludes that the maximal cap $\widetilde{\Sigma}$ is a nonempty.
\end{proof}

\begin{center}
{\sc Acknowledgement}
\end{center}

The  authors would like to thank Professor Jos\'e Arrieta of the Universidad Complutense for helpful discussion.

\end{document}